\documentclass[preprint,12pt]{elsarticle}
\usepackage{amsfonts}
\usepackage{stmaryrd}
\usepackage{graphicx}
\usepackage{epsfig}
\usepackage{xcolor}
\usepackage{amsmath}
\usepackage{amsthm}
\usepackage{amssymb}
\usepackage{mathrsfs}
\usepackage{float}
\usepackage{tikz}
\usepackage{multirow}
\usepackage{verbatim}
\usepackage{fancyhdr}
\usepackage{color}
\usepackage{mathtools}
\usepackage{mathrsfs}
\usepackage{sectsty}
\usepackage[title]{appendix}
\usepackage{threeparttable}
\usepackage{dcolumn}
\usepackage{subcaption}
\usepackage{booktabs}
\usepackage{indentfirst}
\usepackage{setspace}
\usepackage{bm}
\usepackage{enumerate}
\usepackage{lineno}
\usepackage{bbm}
\usepackage[labelfont=bf,labelsep=period]{caption}
\newcolumntype{d}[1]{D{.}{.}{#1}}
\usepackage{hyperref}
\hypersetup{colorlinks,linkcolor=blue,citecolor=blue}

\newtheorem{lemma}{Lemma}[section]

\newtheorem{Remark}{Remark}[section]
\newtheorem{theorem}{Theorem}[section]

\allowdisplaybreaks

\journal{Journal of Computational Physics}

\begin{document}

\begin{frontmatter}

\title{A bound-preserving oscillation-eliminating 
discontinuous Galerkin  
method with operator splitting 
for solving Kapila's five-equation model}

\cortext[cor1]{Corresponding authors. E-mail: yunlongliu@hrbeu.edu.cn, zhangaman@hrbeu.edu.cn}
\author[aff1]{Jia-Jun Zou}

\author[aff2]{Yu-Chang Liu}

\author[aff3]{Fan Zhang}

\author[aff1]{Qi Kong}

\author[aff1]{Yun-Long Liu\corref{cor1}}

\author[aff1]{A-Man Zhang\corref{cor1}}

\affiliation[aff1]{
    organization={College of Shipbuilding Engineering, Harbin Engineering University},
    city={Harbin},
    postcode={150001}, 
    country={P.R. China}
}

\affiliation[aff2]{
    organization={School of Mathematical Sciences, University of Science and Technology of China},
    city={Hefei},
    postcode={230026}, 
    country={P.R. China}  
}

\affiliation[aff3]{
    organization={School of Mathematics and Physics, University of Science and Technology Beijing},
    city={Beijing},
    postcode={100083}, 
    country={P.R. China}
}
\begin{abstract}
This paper proposes a robust 
operator-splitting discontinuous Galerkin (DG) framework 
to overcome the severe stiffness-induced instabilities 
in simulating compressible two-phase flows governed by 
Kapila's five-equation model with the Tammann equation of state. 
Specifically, the system is decoupled into 
a five-equation transport model and a stiff $\kappa$-source term. 
The former is discretized via a 
quasi-conservative DG method \cite{cheng2020quasi}, 
while the latter is resolved 
by the local DG method combined with 
a novel adaptive implicit strategy that 
hybridizes the backward Euler and second-order singly 
diagonally implicit Runge-Kutta schemes. 
This implicit strategy possesses the unconditionally bound-preserving property, 
and thus effectively circumvents the 
severe stability constraints and time-step penalties inherent 
in traditional explicit schemes.
Furthermore, to enhance computational robustness, 
we integrate an oscillation-eliminating DG (OEDG) procedure to suppresses spurious oscillations 
without characteristic decomposition, complemented 
by a bound-preserving limiter to maintain 
physically admissible numerical solutions. 
We also prove that the proposed operator-splitting DG framework, 
integrated with the oscillation-eliminating limiter, and 
the bound-preserving limiter, strictly satisfies the Abgrall condition.
Finally, extensive numerical experiments 
are conducted to demonstrate the superior robustness and efficiency of the method.
\end{abstract}



\end{frontmatter}


\section{Introduction}
Multiphase flows exist in various engineering fields such as aerospace, energy, and environmental sciences. 
Mathematical models for such flows can be classified as either two-fluid models or homogeneous models.
Among the two-fluid models, the Baer-Nunziato model \cite{baer1986two}, which includes mass, momentum, and energy conservation laws for each phase and a volume fraction equation, was first proposed to simulate the deflagration-to-detonation transition in reactive granular materials.
Then, Kapila's five-equation model \cite{kapila2001two} was 
derived as a reduced model under the assumption of 
velocity and pressure equilibrium across phases. 
Compared to the five-equation transport 
model \cite{ALLAIRE2002577,massoni2002proposition}, 
Kapila's model contains an additional 
non-conservative $\kappa$-source term in the volume 
fraction equation that accounts for the different 
compressibility of the two fluids. This enables 
it to achieve higher fidelity in bubble dynamics 
simulation \cite{schmidmayer2020assessment,tiwari2013diffuse,saurel2009simple,beig2018temperatures}.


In recent years, the application of the discontinuous Galerkin (DG) method to multiphase flows has made significant progress. For instance, Cheng et al. \cite{cheng2020quasi} proposed a high-order quasi-conservative DG method to solve the five-equation transport model with the Tammann equation of state (EOS). They subsequently considered Kapila's model based on the ideal gas EOS \cite{zhang2023analysis}, in which the non-conservative source terms are discretized using a simple finite volume framework that achieves only second-order accuracy. 
To address this, Yan et al. \cite{yan2024uniformly} further developed a DG spatial discretization for Kapila's model that successfully achieves optimal high-order accuracy for the source terms. Despite these advances in spatial discretization, simulating gas-liquid two-phase flows by solving Kapila's model remains a challenging problem due to the stiffness introduced by the $\kappa$-source term.

In fact, when a shock wave in water impinges on an air bubble interface, the velocity divergence at the gas–water interface becomes extremely large. This imposes severe restrictions on the allowed time step to maintain bounded volume fractions, often reducing it by several orders of magnitude \cite{zhang2023analysis}. White et al. \cite{white2025high} incorporated the consistent and conservative phase-field formulation into the DG framework to achieve high-order accuracy, where the bound-preserving analysis is also heavily dominated by the stiff $\kappa$-source term. To overcome this stiffness, attempts have been made using alternative approaches. Notable examples include the acoustic-convective splitting-based scheme by Eikelder et al. \cite{ten2017acoustic} and the second-order generalized Riemann problem based finite volume scheme with a semi-implicit time discretization by Chen et al. \cite{chen2025generalized}. However, addressing this issue with high-order DG method has remained unexplored.

In this paper, to circumvent the stiffness-induced 
instability of the $\kappa$-source term, we propose an operator-splitting DG 
framework that 
decomposes the system into a five-equation transport model 
and the $\kappa$-source term. Specifically,
the former is discretized using a quasi-conservative DG method, 
while the latter is discretized using the local DG (LDG) method with
a novel adaptive 
implicit strategy that hybridizes the 
backward Euler and second-order singly 
diagonally implicit Runge-Kutta (SDIRK) schemes, 
ensuring the unconditional bound-preserving property. 
Moreover, we also implement an OEDG 
procedure to suppress numerical oscillations, 
along with an effective bound-preserving limiter 
to ensure the physical constraints of numerical solution. 
Finally, the proposed operator-splitting DG framework, 
integrated with the oscillation-eliminating limiter, and 
the bound-preserving limiter are proved 
to maintain pressure and velocity 
equilibrium at interfaces for the Tammann EOS, 
and its performance is thoroughly 
validated by a series of benchmark 
and challenging test cases in both one and two dimensions. 

The paper is organized as follows.
The governing equations of Kapila's five-equation model are given
in Section \ref{sec-kapila-eq}. The operator-splitting DG 
framework is presented in detail in Section \ref{sec-split}.
The implementation of the BP limiter and OEDG procedure 
is described in Section \ref{sec-limiters}. 
Section \ref{sec-experiments} provides numerical 
results to verify the performance of the proposed framework. 
Finally, Section \ref{sec-conclusion} draws the conclusion remarks.

\section{Governing equations}\label{sec-kapila-eq}

This work considers the immiscible compressible gas-liquid 
two-phase flow described by the following Kapila's 
five-equation model \cite{kapila2001two}
\begin{equation}\label{kapila5eq}
\begin{aligned}
    \frac{\partial}{\partial t}
    \begin{pmatrix}
        z_1\rho_1\\
        z_2\rho_2 \\
        \rho \mathbf{u} \\
        E
    \end{pmatrix}
    &+ \nabla \cdot
    \begin{pmatrix}
        z_1\rho_1 \mathbf{u} \\
        z_2\rho_2 \mathbf{u} \\
        \rho \mathbf{u}^2 + p\mathbf{I}  \\
        \mathbf{u}(E + p) 
    \end{pmatrix}
    = \bm{0}, \\
    \frac{\partial z_1}{\partial t} &+ \mathbf{u}\cdot\nabla z_1 = \kappa\nabla \cdot \mathbf{u}.
\end{aligned}
\end{equation}
Here, $\rho$ denotes the mixture density, $z_k$ the volume fraction of 
phase $k$, $\rho_k$ the density of phase $k$, $p$ the common 
pressure shared by both phases, $\mathbf u$ the velocity, 
$E$ the total energy, and $e$ the specific internal energy of the mixture. 
$\mathbf I$ represents the $d\times d$ identity matrix. The mixture properties are defined as follows
\begin{equation}\label{mixture_proper}
\begin{gathered}
\rho = z_1\rho_1 + z_2\rho_2, \quad 
z_1+z_2=1, \quad 
p = p_1= p_2, \quad 
\rho e = z_1\rho_1 e_1 + z_2\rho_2 e_2.
\end{gathered}
\end{equation}
The parameter $\kappa$ in the volume fraction equation is defined as
\begin{equation}
    \kappa = z_1 z_2 \frac{\nu_1-\nu_2}{\nu},\quad \nu_k = 
    \frac{1}{\rho_k c_k^2},\quad \nu = z_1\nu_1 + z_2\nu_2,
\end{equation}
where $\nu_k$ and $\nu$ are the compressibilities of phase $k$ and the mixture, respectively.
If the $\kappa$-source term $\kappa\nabla \cdot \mathbf{u}$ is omitted, the model 
reduces to the five-equation transport model.
 
The system \eqref{kapila5eq} is closed via the Tammann EOS. 
For each phase, it is given by
\begin{equation}
    p_k = \rho_k e_k(\gamma_k-1)-\gamma_k p_{w,k},
\end{equation}
where $\gamma_k$ is the specific heat ratio and $p_{w,k}$ is the reference pressure for phase $k$. The speed of sound for each phase can be expressed as
\begin{equation}
    c_k = \sqrt{\gamma_k (p+p_{w,k})/\rho_k}.
\end{equation}
To define the mixture parameters, we introduce the following relations
\begin{equation}\label{eq:mixture_parameters}
\frac{1}{\gamma-1}=\frac{z_1}{\gamma_1-1}+\frac{z_2}{\gamma_2-1},\quad 
\frac{\gamma p_w}{\gamma-1}=\frac{z_1\gamma_1 p_{w,1}}{\gamma_1-1}+\frac{z_2\gamma_2 p_{w,2}}{\gamma_2-1},
\end{equation}
where $\gamma$ and $p_w$ are the effective specific heat ratio and reference pressure for the mixture, respectively. Accordingly, the mixture pressure can be written as
\begin{equation}
    p = \rho e(\gamma-1)-\gamma p_{w}.
\end{equation}
The mixture speed of sound for Kapila's model, which obeys Wood's formula, is given by
\begin{equation}
    \frac{1}{\rho c^2} = \frac{z_1}{\rho_1 c_1^2} + \frac{z_2}{\rho_2 c_2^2},
\end{equation}
and the mixture speed of sound for the five-equation transport model is given by
\begin{equation}
    \tilde{c} = \sqrt{\gamma (p+p_{w})/\rho}.
\end{equation}

\begin{Remark}
Consider Kapila's five-equation model and the 
five-equation transport model closed by the Tammann EOS. The two models are acoustically identical (i.e. $ c = \tilde{c}$) in single-phase regions where $\kappa = 0$. In other cases, the relationship between these two acoustic limits is rigorously given by \cite{zhang2023analysis}
\begin{equation}\label{eq:sound_speed_relation}
    \rho c^2 = \rho \tilde{c}^2 - \kappa (\gamma - 1)(\rho_2 e_2 - \rho_1 e_1).
\end{equation}
\end{Remark}

The system \eqref{kapila5eq} can also be written concisely as
\begin{equation}\label{eq:compact_form_kapila}
\begin{aligned}
\frac{\partial \mathbf{W}}{\partial t}&+\nabla \cdot \mathbf{F}(\mathbf W,z_1)= 0,\\
\frac{\partial z_1}{\partial t} &+ \mathbf{u}\cdot\nabla z_1 = \kappa\nabla \cdot \mathbf{u},
\end{aligned}
\end{equation}
where $\mathbf W=(z_1\rho_1,z_2\rho_2,\rho\mathbf u, E)^\mathrm{T}$, $\mathbf F = ( z_1\rho_1 \mathbf{u},
        z_2\rho_2 \mathbf{u},
        \rho \mathbf{u}^2 + p\mathbf{I},
        \mathbf{u}(E + p) )^\mathrm{T}$.
In addition, we define $\mathbf U=(z_1\rho_1,z_2\rho_2,\rho\mathbf u, E, z_1)^\mathrm{T}$.

\section{Operator-splitting DG framework}\label{sec-split}
\subsection{Operator-splitting approach for Kapila's model}
When shocks and rarefaction waves hit a material interface, 
the $\kappa$-source term may become a stiff term in the volume fraction equation. 
To address this, we split Kapila's model into 
two parts. The first part considers only the five-equation transport model, 
and the second part 
considers the $\kappa$-source term. Specifically, the first part is

\begin{equation}\label{eq:RHS}
\mathbf{U}_t = \text{RHS}(\mathbf{U}),
\end{equation}
where the expression of $\text{RHS}(\mathbf{U})$ is given by

\begin{equation}
\text{RHS}(\mathbf{U}) = 
    \begin{pmatrix}
       -\nabla \cdot z_1\rho_1 \mathbf{u} \\
       -\nabla \cdot z_2\rho_2 \mathbf{u} \\
       -\nabla \cdot \rho \mathbf{u}^2 - p\mathbf{I}  \\
       -\nabla \cdot \mathbf{u}(E + p) \\
       -\mathbf{u} \cdot \nabla z_1
    \end{pmatrix}.
\end{equation}
And the second part is
\begin{equation}\label{kappa}
\mathbf{U}_t = \mathbf{S}(\mathbf{U}),
\end{equation}
where the expression of $\mathbf{S}(\mathbf{U})$ is given by
\begin{equation}
\mathbf{S}(\mathbf{U})=
    \begin{pmatrix}
       0 \\
       0 \\
       \mathbf{0}  \\
       0 \\
       \kappa\nabla\cdot\mathbf u.
    \end{pmatrix}
\end{equation}

In this work, we use the second-order Strang splitting for Kapila's model \cite{strang1968construction}
\begin{align*}
&\text{Step 1 : } \frac{\mathbf{U}_1 - \mathbf{U}^n}{\Delta t_1} = \mathbf{S}(\mathbf{U}_1), & &\Delta t_1 = \Delta t/2, && \mathbf{U}_{1,0} = \mathbf{U}^n, \\
&\text{Step 2 : } \frac{\mathbf{U}_2 - \mathbf{U}_1}{\Delta t_2} = \text{RHS}(\mathbf{U}_1), & &\Delta t_2 = \Delta t, && \mathbf{U}_{2,0} = \mathbf{U}_1, \\
&\text{Step 3 : } \frac{\mathbf{U}^{n+1} - \mathbf{U}_2}{\Delta t_3} = \mathbf{S}(\mathbf{U}^{n+1}), & &\Delta t_3 = \Delta t/2, && \mathbf{U}_0^{n+1} = \mathbf{U}_2,
\end{align*}
where the superscript `$n$' denotes the time instant $t = t^n$, and the subscript 0 of $\mathbf{U}_0$, represents the initial condition of $\mathbf{U}$. 
Here, Step 1 and Step 3 correspond to 
implicitly solving the $\kappa$-source term, 
while Step 2 corresponds to explicitly solving 
the five-equation transport equation. 
We note that our framework will be a BP-OEDG method for the five-equation transport model when we omit Step 1 and Step 3. 
In Step 2, high order time discretization method 
is needed to achieve high-order temporal accuracy. 
Notably, Strang splitting achieves a good equilibrium between accuracy and 
efficiency \cite{strang1968construction}, and we refer to \cite{sheng1989solving,suzuki1991general} for higher-order operator-splitting methods.

\begin{Remark}
    Although the second-order Strang 
    splitting is limited to $O(\Delta t^2)$ 
    temporal accuracy, this constraint 
    primarily applies to two-phase flow 
    regions. In single-phase flow regions, 
    since Step 1 and Step 3 can be bypassed, 
    the algorithm effectively reduces to Step 2 alone, and thus preserves high-order fidelity in the smooth, single-phase regions.
\end{Remark}

\subsection{quasi-conservative DG for the five-equation transport model}
In this work, we use the quasi-conservative DG spatial discretization \cite{cheng2020quasi}  for solving the five-equation transport model. 
For the sake of completeness, we give a brief review of the quasi-conservative DG method for the five-equation transport model.

We first introduce some basic notation. Let $\Omega$ be the computational domain, which is 
composed of $N_x \times N_y$ uniform Cartesian elements $I_{i,j} = [x_{i-\frac{1}{2}}, x_{i+\frac{1}{2}}] \times [y_{j-\frac{1}{2}}, y_{j+\frac{1}{2}}]$ for 
$1 \le i \le N_x$ and $1 \le j \le N_y$. The mesh is denoted as $ \Omega_h=\{I_{ij}\}$, and the element size is
$\Delta x = x_{i+\frac{1}{2}}-x_{i-\frac{1}{2}}$ and $\Delta y = y_{j+\frac{1}{2}}-y_{j-\frac{1}{2}}$.
 The discontinuous finite element space is defined as follows
\begin{equation}
V_h^K := \{ v_h \in L^2(\Omega) : v_h|_{I_{i,j}} \in P^K(I_{i,j}), \, \forall I_{i,j} \in \Omega_h \}.
\end{equation}
Then, we can express the DG($P^K$) approximate solution of the vector
 $\mathbf U$ as follows
\begin{equation}
\mathbf{U}_h(x, y) = \sum_{m=0}^{C^K - 1} \mathbf{U}_{i,j}^{(m)} \varphi_{i,j}^{(m)}(x,y), \quad \forall (x, y) \in I_{i,j},
\end{equation}
where \( \mathbf{U}_h(x,y) \) is the approximate solution.
\( \varphi_{i,j}^{(m)} \in P^K(I_{i,j}) \) and \( \mathbf{U}_{i,j}^{(m)} \) are 
the $m$-th basis function and the corresponding coefficient, respectively, 
and \( C^K \) is the total number of coefficients. In this work, 
we use the Legendre local orthogonal basis functions which result 
in the following cell-averaged approximate solution
\begin{equation}
\overline{\mathbf{U}}_{i,j} := \frac{1}{|I_{i,j}|} \int_{I_{i,j}} \mathbf{U}_h \, \mathrm dx \mathrm dy = \mathbf{U}_{i,j}^{(0)}.
\end{equation}
Because of the discontinuous nature of variables on element interfaces $\mathcal F$,
we need to define $\mathbf U ^-$ and $\mathbf U ^+$ on element interfaces $\mathcal F$, 
which means the solutions taken from the left/lower and right/upper element, 
respectively, see Fig.~\ref{fig:combined} (\subref{fig:GL}) for
an illustration, where the red points are Gauss-Lobatto points.

For five-equation transport model, its compact form as follows:
\begin{equation}\label{eq:compact_form_Allaire}
\begin{aligned}
\frac{\partial \mathbf{W}}{\partial t}+\nabla \cdot \mathbf{F}(\mathbf W,z_1)= 0,\\
\frac{\partial z_1}{\partial t} + \mathbf{u}\cdot\nabla z_1 = 0.
\end{aligned}
\end{equation}
For the conservative variables $\mathbf W$, the spatial
 discretization in element $I_{i,j}$ is given by
\begin{equation}
   \int\limits_{I_{i,j}}{\frac{\partial \mathrm{\mathbf{W}}}{\partial t}\varphi_{i,j}^m \, \mathrm dS} = 
   \int\limits_{I_{i,j}} \mathbf{F}(\mathrm{\mathbf{W}},z_1)  \cdot \nabla \varphi_{i,j}^m \, \mathrm dS-
   \int\limits_{\partial I_{i,j}}{ \mathbf{\hat{F}}(\mathrm{\mathbf{W}},z_1)   \cdot \mathbf{n} \varphi_{i,j}^m \, \mathrm dl}.
\label{eq:Allaire-w}
\end{equation}
Denoting $\mathbf F=(\mathbf F^1,\mathbf F^2)$, the numerical flux is determined by the following Lax-Friedrichs numerical flux
\begin{equation}
\mathbf{\hat{F}}^d(\mathbf{W},z_1) = \frac{1}{2} \left( \mathbf{F}^d(\mathbf{W^+},z_1^+) + \mathbf{F}^d(\mathbf{W^-},z_1^-) \right) - \frac{S_d}{2} (\mathbf{W^+} - \mathbf{W^-}),
\end{equation}
where $d$= 1 or 2, depending on the direction of the flux, and $S_d$ is the estimate of the characteristic speed defined as
\begin{equation} 
 S_1 = \max\left\{ \lvert u^+ \rvert + \tilde{c}^+, \lvert u^- \rvert + \tilde{c}^- \right\}, \quad S_2 = \max\left\{ \lvert v^+ \rvert + \tilde{c}^+, \lvert v^- \rvert + \tilde{c}^- \right\} .
 \end{equation}

For the non-conservative variable $z_1$, the quasi-conservative 
spatial discretization in element $I_{i,j}$ is given as
\begin{equation}
\begin{aligned}
\int\limits_{I_{i,j}}{\frac{\partial z_1}{\partial t}\varphi_{i,j}^m \, \mathrm dS} 
= - \int\limits_{I_{i,j}}\varphi_{i,j}^m \mathbf{u} \cdot \nabla z_1 \, \mathrm dS-
\int\limits_{\partial I_{i,j}} \varphi_{i,j}^m 
\widehat{(\mathbf{u} \cdot \mathbf{n}z_1) } \, \mathrm dl +
\int\limits_{\partial I_{i,j}^{in}} \varphi_{i,j}^m(\mathbf{u}  \cdot 
\mathbf{n}z_1) \, \mathrm dl,
\end{aligned}
\label{eq:Allaire-z}
\end{equation}
where $\partial I_{i,j}$ and $\partial I_{i,j}^{in}$ are 
the boundary and inner boundary of element $I_{i,j}$, 
respectively, see Fig.~\ref{fig:combined} (\subref{fig:inner}) for
an illustration.

\begin{figure}[htbp]
    \centering
    \begin{subfigure}[b]{0.48\textwidth}
        \centering
        \begin{tikzpicture}[scale=1.5]
            
            \draw[thick] (0,0) rectangle (3,3);
            
            \def\gla{0} 
            \def\glb{0.82918} 
            \def\glc{2.17082} 
            \def\gld{3} 
            
            \foreach \x in {\gla, \glb, \glc, \gld} {
                \foreach \y in {\gla, \glb, \glc, \gld} {
                    \fill[red] (\x, \y) circle (1.5pt); 
                }
            }
            
            \node[below] at (0, -0.2) {$x_{i-\frac{1}{2}}$};
            \node[below] at (3, -0.2) {$x_{i+\frac{1}{2}}$};
            \node[left]  at (-0.2, 0) {$y_{j-\frac{1}{2}}$};
            \node[left]  at (-0.2, 3) {$y_{j+\frac{1}{2}}$};
            
            \node at (-0.25, 1.5) {$-$}; 
            \node at (0.15, 1.5) {$+$}; 
            \node at (2.85, 1.5) {$-$}; 
            \node at (3.25, 1.5) {$+$};  
            
            \node at (1.5, -0.25) {$-$}; 
            \node at (1.5, 0.15) {$+$};  
            \node at (1.5, 2.85) {$-$}; 
            \node at (1.5, 3.25) {$+$};  

        \end{tikzpicture}
        \caption{}
        \label{fig:GL}
    \end{subfigure}
    \hfill
    \begin{subfigure}[b]{0.48\textwidth}
        \centering
        \begin{tikzpicture}[scale=1.5]
            
            \draw[thick] (0,0) rectangle (3,3);
            
            \draw[thick, dashed, blue] (0.1, 0.1) rectangle (2.9, 2.9);
            
            \node[below] at (0, -0.1) {$x_{i-\frac{1}{2}}$};
            \node[below] at (3, -0.1) {$x_{i+\frac{1}{2}}$};
            \node[left]  at (-0.1, 0) {$y_{j-\frac{1}{2}}$};
            \node[left]  at (-0.1, 3) {$y_{j+\frac{1}{2}}$};
            
            \node[above] at (1.5, 3) {$\partial I_{i,j}$};
            \node[below, blue] at (1.5, 2.8) {$\partial I_{i,j}^{in}$};

        \end{tikzpicture}
        \caption{}
        \label{fig:inner}
    \end{subfigure}
    
    \caption{Element $I_{i,j}$ illustrations: (a) 4x4 Gauss-Lobatto
     points (b) boundary $\partial I_{i,j}$ and inner boundary $\partial I_{i,j}^{in}$}
    \label{fig:combined}
\end{figure}

The integral over the inner boundary of element $I_{i,j}$ is discretized by:
\begin{equation}
\begin{aligned}
&\int\limits_{\partial I_{i,j}^{in}} \varphi_{i,j}^m (\mathbf{u} \cdot \mathbf{n} z_1) \, \mathrm dl  \\                                                  
&=\int_{y_{j-\frac{1}{2}}}^{y_{j+\frac{1}{2}}} \left[ (u z_1) \varphi_{i,j}^m \right]_{x=x_{i+\frac{1}{2}}^{-}} \, \mathrm dy  
 - \int_{y_{j-\frac{1}{2}}}^{y_{j+\frac{1}{2}}} \left[ (u z_1) \varphi_{i,j}^m \right]_{x=x_{i-\frac{1}{2}}^{+}} \, \mathrm dy \\ 
& + \int_{x_{i-\frac{1}{2}}}^{x_{i+\frac{1}{2}}} \left[ (v z_1) \varphi_{i,j}^m \right]_{y=y_{j+\frac{1}{2}}^{-}} \, \mathrm dx 
 - \int_{x_{i-\frac{1}{2}}}^{x_{i+\frac{1}{2}}} \left[ (v z_1) \varphi_{i,j}^m \right]_{y=y_{j-\frac{1}{2}}^{+}} \,\mathrm dx.
\end{aligned}
\end{equation}
The integration of numerical flux $\widehat{(\mathbf{u} \cdot \mathbf{n}z_1)}$ in element
 boundary $\partial I_{i,j}$ is expanded as:
\begin{equation}
\begin{aligned}
&\int\limits_{\partial I_{i,j}} \varphi_{i,j}^m \widehat{(\mathbf{u} \cdot \mathbf{n} z_1)} \, \mathrm dl  \\                                                  
&=\int_{y_{j-\frac{1}{2}}}^{y_{j+\frac{1}{2}}} \left[ \widehat{u z_1} \, \varphi_{i,j}^m \right]_{x=x_{i+\frac{1}{2}}} \, \mathrm dy  
 - \int_{y_{j-\frac{1}{2}}}^{y_{j+\frac{1}{2}}} \left[ \widehat{u z_1} \, \varphi_{i,j}^m \right]_{x=x_{i-\frac{1}{2}}} \, \mathrm dy \\ 
& + \int_{x_{i-\frac{1}{2}}}^{x_{i+\frac{1}{2}}} \left[ \widehat{v z_1} \, \varphi_{i,j}^m \right]_{y=y_{j+\frac{1}{2}}} \, \mathrm dx 
 - \int_{x_{i-\frac{1}{2}}}^{x_{i+\frac{1}{2}}} \left[ \widehat{v z_1} \, \varphi_{i,j}^m \right]_{y=y_{j-\frac{1}{2}}} \, \mathrm dx,
\end{aligned}
\end{equation}
where the specific
expression of numerical flux about element $I_{i,j}$ is as follows
\begin{equation}
\begin{aligned}
&\widehat{u z_1} = \frac{1}{2} u^{in} \big( z_1^- + z_1^+ \big) - \frac{1}{2} S_1 \big( z_1^+ - z_1^- \big),  \\
&\widehat{v z_1} = \frac{1}{2} v^{in} \big( z_{1}^- + z_1^+ \big) - \frac{1}{2} S_2 \big( z_1^+ - z_1^- \big),
\end{aligned}
\end{equation}
where the superscript "in" denotes the interior of the current element.

The quasi-conservative spatial discretization preserves conservation of the mass of each fluid, total momentum and total
energy and satisfies the Abgrall condition \cite{cheng2020quasi}. Moreover, it can recover
 the first-order quasi-conservative finite volume scheme \cite{ALLAIRE2002577} when we use DG($P^0$) spatial discretization.
Although the positivity-preserving condition 
for internal energy was established 
in \cite{cheng2020quasi}, the strict 
positivity of internal energy 
does not guarantee a positive pressure under 
the Tammann EOS. Despite this theoretical limitation, 
we have not encountered any instabilities during numerical experiments.

\subsection{Implicit LDG method for the $\kappa$-source term}
\subsubsection{Implicit formulation}
Let us focus on solving the second part \eqref{kappa}. The partial densities, 
momentum, and total energy remain unchanged, and only volume fraction $z_1$ is updated. In this work, we discretize this equation implicitly 
using the first-order backward Euler method
\begin{equation}\label{imkappa}
        \frac{z_1^{n+1}-z_1^{n}}{\Delta t_{\mathrm{im}}} = \kappa^{n+1}\nabla \cdot \mathbf{u}^{n} =: f(z_1^{n+1}), 
\end{equation}
where $\Delta t_{\mathrm{im}} = {\Delta t}/{2}$, and
\begin{equation}
    \begin{aligned}
        \kappa^{n+1} &= z_1^{n+1} z_2^{n+1} \frac{\nu_1^{n+1}-\nu_2^{n+1}}{\nu^{n+1}}, \quad &
        \nu^{n+1} &= z_1^{n+1}\nu_1^{n+1} + z_2^{n+1}\nu_2^{n+1}, \\
        \nu_k^{n+1} &= \frac{1}{\gamma_k (p^{n+1}+p_{w,k})}, \quad &
        p^{n+1} &= (\rho e)^{n}(\gamma^{n+1}-1)-\gamma^{n+1} p_{w}^{n+1}.
    \end{aligned}
\end{equation}
Here, the mixture 
parameters $\gamma^{n+1}$ and $p_{w}^{n+1}$ are functions of 
$z_1^{n+1}$, and satisfy the 
equation \eqref{eq:mixture_parameters}. We can prove the bound-preserving property of the implicit scheme.

\begin{theorem}\label{thm:implicit_solver}
Given $z_1^n \in (0, 1)$ and $\Delta t_{\mathrm{im}} > 0$, 
the updated volume fraction $z_1^{n+1}$ defined by the implicit scheme \eqref{imkappa} admits a solution such that $z_1^{n+1} \in (0, 1)$. 
\end{theorem}

\begin{proof}
First, we define the auxiliary function $F(z_1)$ as
\begin{equation}
    F(z_1) = \frac{z_1 - z_1^n}{\Delta t_{\mathrm{im}}} - f(z_1).
\end{equation}
Then, solving \eqref{imkappa} is equivalent to finding the root of $F(z_1) = 0$.  Considering the source term $f(z_1)$, we observe that 
\begin{equation}
    f(0) = 0, \quad f(1) = 0.
\end{equation}
Evaluating $F(z_1)$ at the interval endpoints yields
\begin{align}
    F(0) = \frac{ - z_1^n}{\Delta t_{\mathrm{im}}} - f(0) = \frac{ - z_1^n}{\Delta t_{\mathrm{im}}}, \quad
    F(1) = \frac{1 - z_1^n}{\Delta t_{\mathrm{im}}} - f(1) = \frac{1 - z_1^n}{\Delta t_{\mathrm{im}}}.
\end{align}
Since $z_1^n \in (0, 1)$, it follows that
\begin{equation}
    F(0) < 0, \quad  F(1) > 0.
\end{equation}
The continuity of $F(z_1)$ guarantees the 
existence of root $z_1^{n+1} \in (0, 1)$. 

In practice, the root of $F(z_1)$ can be efficiently located via the bisection method at every quadrature point. 
Then, the approximate solution $z_1^{n+1}$ can be obtained via an $L^2$ projection. This completes the proof.
\end{proof}

When the second-order Strang splitting is employed, 
the temporal splitting error is $O(\Delta t^2)$. 
To ensure that the final numerical solution genuinely 
achieves an overall second-order (or higher) accuracy, 
the temporal discretization schemes for the subsystems \eqref{kappa} 
must also be at least second-order accurate \cite{leveque2002finite,toro2013riemann}. 
Otherwise, the global error may be dominated 
by the lower-order subsystem errors.
To achieve higher temporal accuracy while maintaining the bound-preserving property, 
we propose an adaptive implicit strategy which
 hybridizes the backward Euler and SDIRK2 method \cite{alexander1977diagonally} as follows
\begin{equation}\label{adaptive_sdirk2}
\begin{aligned}
    z_1^{*} &= z_1^{n} + J \Delta t_{\mathrm{im}} f(z_1^{*}), \\ \tilde{z}_{\mathrm{pred}} &= \left( 2 - \frac{1}{J} \right) z_1^{n} + \left( \frac{1}{J} - 1 \right) z_1^{*}, \\
    z_1^{n+1} &=
    \begin{cases}
        \tilde{z}_{\mathrm{pred}} + J \Delta t_{\mathrm{im}} f(z_1^{n+1}), & \text{if } \tilde{z}_{\mathrm{pred}} \in (0, 1),\\[2ex]
        z_1^{*} + (1 - J) \Delta t_{\mathrm{im}} f(z_1^{n+1}), & \text{otherwise},
    \end{cases}
\end{aligned}
\end{equation}
where $J= 1-\frac{\sqrt{2}}{2}$ is a constant. Specifically, 
the scheme initiates 
with a backward Euler scheme over a time step $J \Delta t_{\mathrm{im}}$, 
which strictly coincides with the first stage of the standard SDIRK2 method. 
If the explicitly extrapolated predictor $\tilde{z}_{\mathrm{pred}}$ 
falls within the interval $(0, 1)$, 
the algorithm proceeds with the 
second stage of SDIRK2 to maintain second-order accuracy; 
otherwise, it locally degenerates 
into a robust backward Euler scheme over the 
remaining $(1 - J) \Delta t_{\mathrm{im}}$ interval to strictly enforce the bound-preserving property. 
For both stages in equation \eqref{adaptive_sdirk2}, we employ the bisection method.

\begin{theorem} \label{im_abgrall}
The implicit solver \eqref{imkappa} satisfies the Abgrall condition, which dictates that 
uniform pressure and velocity fields must be 
preserved across an isolated two-phase material interface, namely
\begin{equation}
\text{if } \mathbf u^n(x,y) \equiv \mathbf u_0, \quad p^n(x,y) \equiv p_0, \quad \text{then } \mathbf u^{n+1}(x,y) \equiv \mathbf u_0, \quad p^{n+1}(x,y) \equiv p_0.
\label{eq:abgrall}
\end{equation}
\end{theorem}

\begin{proof}
  When $\mathbf{u}^n(x,y) \equiv \mathbf{u}_0$, 
  the velocity divergence vanishes ($\nabla \cdot \mathbf{u}^n = 0$). 
  Substituting this into the implicit update \eqref{imkappa} 
  yields $z_1^{n+1} = z_1^n$. Since the partial densities, 
  momentum, and total energy are explicitly held constant 
  during this substep, the entire state vector remains 
  unchanged ($\mathbf{U}^{n+1} = \mathbf{U}^n$). 
  Consequently, both the velocity and pressure fields are preserved. This completes the proof.
\end{proof}

\subsubsection{Reconstruction for $\nabla \cdot \mathbf{u}$}
When employing the DG($P^K$) method for 
the state vector $\mathbf{U}$, the expected spatial 
accuracy for the velocity field $\mathbf{u}$ is 
$\mathcal{O}(\Delta x^{K+1})$. However, 
direct local differentiation to compute $\nabla \cdot \mathbf{u}$ 
results in a sub-optimal accuracy of $\mathcal{O}(\Delta x^K)$. 
To recover the optimal $(K+1)$-th order accuracy for the divergence 
term, inspired by the LDG method \cite{cockburn1998local,cockburn1989tvb}, we employ the 
following reconstruction procedure to 
obtain an approximation $D_h \approx \nabla \cdot \mathbf{u}$. 

Let $D_h \in V_h^K$. 
Instead of directly differentiating the polynomials 
inside each element, we seek a weak formulation for $D_h$. 
By multiplying the divergence with 
the test function $\varphi_{i,j}^m$ and performing integration 
by parts over the element $I_{i,j}$, 
we obtain:
\begin{equation}\label{eq:reconstruction_Dh}
    \int_{I_{i,j}} D_h \varphi_{i,j}^m   \mathrm dS =\int_{\partial I_{i,j}} 
    \widehat{(\mathbf{u} \cdot \mathbf{n})} \varphi_{i,j}^m  
    \, \mathrm dl  - \int_{I_{i,j}} \mathbf{u} \cdot \nabla \varphi_{i,j}^m  \, \mathrm dS.
\end{equation}
To physically respect the advection direction and maintain numerical stability, 
we specify a purely upwind numerical flux. Specifically, 
the upwind traces are determined by the local flow directions 
at the respective interfaces:
\begin{equation}
    \widehat{u} = u^{\text{up}} = 
    \begin{cases} 
        u^-, & \text{if } \bar{u} \ge 0, \\
        u^+, & \text{if } \bar{u} < 0 ,
    \end{cases}
    \quad \quad 
    \widehat{v} = v^{\text{up}} = 
    \begin{cases} 
        v^-, & \text{if } \bar{v} \ge 0, \\
        v^+, & \text{if } \bar{v} < 0, 
    \end{cases}
\end{equation}
where $\bar{u}$ and $\bar{v}$ represent the 
characteristic advection speeds at the interface. 
In this work, we employ a simple arithmetic mean (e.g., $\bar{u}=(u^- + u^+)/2$). 
Once the reconstructed 
divergence $D_h^n$ is obtained 
at time level $n$, we replace 
the $\nabla \cdot \mathbf{u}^n$ 
with $D_h^n$ in the source 
term $f(z_1)$ within the 
implicit scheme \eqref{imkappa}, thereby restoring the optimal accuracy.

\section{Nonlinear limiters and computational procedure}\label{sec-limiters}
\subsection{Oscillation-eliminating procedure}
The semi-discrete DG method \eqref{eq:Allaire-w} and \eqref{eq:Allaire-z} can be rewritten in 
an ODE form, which can be discretized 
by the SSP-RK scheme \cite{shu1988total,gottlieb2001strong,zhu2013runge}. To avoid the potential numerical oscillations, we use an oscillation-eliminating procedure $\mathcal F_{{\Delta t}} \mathbf U_h = \mathbf U_h^{\sigma}$  after 
each SSP-RK stage, where $F_{{\Delta t}}$ 
represents the oscillation-eliminating operator defined as \cite{peng2025oedg}
\begin{equation}
\mathcal{F}_{\Delta t}\mathbf{U}_h = \mathbf U_h^{\sigma} = \mathbf{U}_{i,j}^{(0)}\varphi_{i,j}^{(0)}(x,y) + \sum_{k=1}^K e^{-\Delta t\sum_{m=0}^k \delta_{i,j}^m(\mathbf{U}_h)} \sum_{|\boldsymbol{\alpha}|=k} \mathbf{U}_{i,j}^{(\boldsymbol{\alpha})}\varphi_{i,j}^{(\boldsymbol{\alpha})}(x,y).
\label{oe-operator}
\end{equation}
Here, $\Delta t$ is the time step size, $\boldsymbol{\alpha}=(\alpha_1, \alpha_2)$ is the multi-index with $|\boldsymbol{\alpha}| = \alpha_1 + \alpha_2$. The coefficient $\delta_{i,j}^m(\mathbf{U}_h)$ is defined as 
\begin{equation}
\delta_{i,j}^m(\mathbf{U}_h) = \max_{1 \le q \le 6} \left( \frac{\beta_{i,j}^x \big( \sigma_{i+\frac{1}{2},j}^m(\mathbf{U}_h^{(q)}) + \sigma_{i-\frac{1}{2},j}^m(\mathbf{U}_h^{(q)}) \big)}{\Delta x} 
 + \frac{\beta_{i,j}^y \big( \sigma_{i,j+\frac{1}{2}}^m(\mathbf{U}_h^{(q)}) + \sigma_{i,j-\frac{1}{2}}^m(\mathbf{U}_h^{(q)}) \big)}{\Delta y} \right),
\end{equation}
with
\begin{equation}
\sigma_{i+\frac{1}{2},j}^m(\mathbf{U}_h^{(q)})
= \begin{cases}
0, & \text{if } \mathbf{U}_h^{(q)} \equiv \overline{\mathbf{U}}_\Omega^{(q)}, \\[2ex]
\dfrac{(2m+1)\Delta x^m}{2(2K-1)m!} \displaystyle\sum_{|\boldsymbol{\alpha}|=m} \frac{ \frac{1}{\Delta y} \int_{y_{j-\frac{1}{2}}}^{y_{j+\frac{1}{2}}} \big| [\![ \partial^{\boldsymbol{\alpha}} \mathbf{U}_h^{(q)} ]\!]_{i+\frac{1}{2},j} \big| \mathrm{d}y }{ \big\| \mathbf{U}_h^{(q)} - \overline{\mathbf{U}}_\Omega^{(q)} \big\|_{L^\infty(\Omega)} }, & \text{otherwise},
\end{cases}
\label{sigmayou}
\end{equation}
and
\begin{equation}
\sigma_{i,j+\frac{1}{2}}^m(\mathbf{U}_h^{(q)})
= \begin{cases}
0, & \text{if } \mathbf{U}_h^{(q)} \equiv \overline{\mathbf{U}}_\Omega^{(q)}, \\[2ex]
\dfrac{(2m+1)\Delta y^m}{2(2K-1)m!} \displaystyle\sum_{|\boldsymbol{\alpha}|=m} \frac{ \frac{1}{\Delta x} \int_{x_{i-\frac{1}{2}}}^{x_{i+\frac{1}{2}}} \big| [\![ \partial^{\boldsymbol{\alpha}} \mathbf{U}_h^{(q)} ]\!]_{i,j+\frac{1}{2}} \big| \mathrm{d}x }{ \big\| \mathbf{U}_h^{(q)} - \overline{\mathbf{U}}_\Omega^{(q)} \big\|_{L^\infty(\Omega)} }, & \text{otherwise},
\end{cases}
\label{sigmashang}
\end{equation}
where $\beta_{i,j}^x$ and $\beta_{i,j}^y$ are the estimates of the local 
maximum wave speeds in the $x$- and $y$-directions, respectively, 
$\overline{\mathbf{U}}_\Omega^{(q)} = \frac{1}{|\Omega|} \int_{\Omega} \mathbf{U}_h^{(q)}(x,y)\mathrm{d}x\mathrm{d}y$ 
represents the global average of the $q$th variable $\mathbf{U}_h^{(q)}$ over the entire computational domain $\Omega$.

Next, we consider the Abgrall condition of the oscillation-eliminating procedure. As a preliminary point, we first introduce the following two lemmas.

\begin{lemma}\label{line-invar}
If the numerical solution satisfies a linear identity $\mathbf{\Lambda}\mathbf{U}_h(x,y) \equiv \mathbf{c}$ for a given $\mathbf{\Lambda} \in \mathbb{R}^{q \times 6}$, $\mathbf{c} \in \mathbb{R}^{q \times 1}$, and $q \in \mathbb{Z}^+$, then it follows that \cite{yan2024uniformly}
\begin{equation}\label{eq:inv}
\mathbf{\Lambda}\mathcal{F}_{\Delta t}(\mathbf{U}_h(x,y)) \equiv \mathbf{c}, 
\end{equation}
which implies the oscillation-eliminating operator is linearity-invariant.
\end{lemma}

\begin{lemma}\label{2d-quasi-abgrall}
The semi-discrete DG method \eqref{eq:Allaire-w} and \eqref{eq:Allaire-z} coupled with the forward Euler time discretization maintains the Abgrall condition \eqref{eq:abgrall} around an isolated material interface, that is \cite{cheng2020quasi}
\begin{equation*}
    \text{if } \mathbf u_\sigma^n(x,y) \equiv \mathbf u_0, \quad p_\sigma^n(x,y) \equiv p_0, \quad \text{then } \mathbf u^{n+1}(x,y) \equiv \mathbf u_0, \quad p^{n+1}(x,y) \equiv p_0.
\end{equation*}
\end{lemma}

Finally, we establish the Abgrall condition of OEDG methood by Theorem \ref{OEDG_Abgrall}.
\begin{theorem} \label{OEDG_Abgrall}
The fully-discrete OEDG method with an SSP-RK scheme maintains the Abgrall condition around an isolated material interface, that is
\begin{equation*}
    \text{if } \mathbf{u}_\sigma^n(x,y) \equiv \mathbf{u}_0, \quad p_\sigma^n(x,y) \equiv p_0, \quad \text{then } \mathbf{u}_\sigma^{n+1}(x,y) \equiv \mathbf{u}_0, \quad p_\sigma^{n+1}(x,y) \equiv p_0,
\end{equation*}
where $\mathbf{u}_\sigma^{n+1} := \mathbf{u}(\mathcal{F}_{\Delta t}\mathbf{U}_h^{n+1}(x,y))$ and $p_\sigma^{n+1} := p(\mathcal{F}_{\Delta t}\mathbf{U}_h^{n+1}(x,y))$.
\end{theorem}

\begin{proof}
Since an SSP RK method can be considered as a convex combination of
the forward Euler method, the RK step in our fully-discrete OEDG scheme preserves the Abgrall condition around an isolated material interface according to
Lemma \ref{2d-quasi-abgrall}. Hence, we only need to prove that the OE procedure maintains the 
Abgrall condition, that is
\begin{equation*}
    \text{if } \mathbf{u}^n(x,y) \equiv \mathbf{u}_0, \quad p^n(x,y) \equiv p_0,
     \quad \text{then } \mathbf{u}_\sigma^{n}(x,y) \equiv \mathbf{u}_0, \quad p_\sigma^{n}(x,y) \equiv p_0.
\end{equation*}

First, based on the scale-invariant property \eqref{eq:inv}, the condition  $\mathbf{u}^n(x,y) \equiv \mathbf{u}_0$ gives
\begin{equation*}
    (\rho \mathbf{u})_\sigma^n(x,y) = \mathcal{F}_{\Delta t}(\rho \mathbf{u})^n(x,y) = \mathbf{u}_0 \mathcal{F}_{\Delta t}\rho^n(x,y) = \mathbf{u}_0 \rho_\sigma^n(x,y),
\end{equation*}
which implies
\begin{equation} \label{eq:vel_const}
    \mathbf{u}_\sigma^n(x,y) = (\rho \mathbf{u})_\sigma^n(x,y) / \rho_\sigma^n(x,y) \equiv \mathbf{u}_0.
\end{equation}
Thus, uniform velocity is preserved by the OEDG method.

Next, under the conditions $\mathbf{u}^n(x,y) \equiv \mathbf{u}_0$ and $p^n(x,y) \equiv p_0$, it holds
\begin{equation*}
    E^n(x,y) - \frac{1}{2} |\mathbf{u}_0|^2 \rho^n(x,y) = p_0 \sum_{\ell=1}^2 \frac{z_\ell^n}{\gamma_\ell - 1}+\sum_{\ell=1}^2 \frac{z_\ell^n \gamma_\ell p_{w,\ell}}{\gamma_\ell - 1},
\end{equation*}
which can be expressed as $\boldsymbol{\Lambda} \mathbf{U}^n(x) = \mathbf c = [0,0,0,0,0,c]^\mathrm{T}$ with
\begin{equation*}
    \boldsymbol{\Lambda} = \left( -\frac{1}{2} |\mathbf{u}_0|^2, -\frac{1}{2} |\mathbf{u}_0|^2, 0, 0, 1, 
    \frac{p_0+\gamma_2 p_{w,2}}{\gamma_2 - 1} - \frac{p_0+\gamma_1 p_{w,1}}{\gamma_1 - 1}  \right), \quad c = \frac{p_0+\gamma_2 p_{w,2}}{\gamma_2 - 1}.
\end{equation*}
Again, thanks to the linearity-invariant property \eqref{eq:inv} of the OE procedure, we obtain $\boldsymbol{\Lambda} \mathcal{F}_{\Delta t}(\mathbf{U}^n(x)) = \mathbf c$, which implies
\begin{equation*}
    \mathcal{F}_{\Delta t}E^n(x,y) - \frac{1}{2} |\mathbf{u}_0|^2 \mathcal{F}_{\Delta t}\rho^n(x,y) = p_0 \sum_{\ell=1}^2 \frac{\mathcal{F}_{\Delta t}z_\ell^n}{\gamma_\ell - 1}+\sum_{\ell=1}^2 \frac{\mathcal{F}_{\Delta t}z_\ell^n \gamma_\ell p_{w,\ell}}{\gamma_\ell - 1}.
\end{equation*}
Using \eqref{eq:vel_const}, this equation can be further expressed as
\begin{equation*}
    E_\sigma^n - \frac{1}{2}\rho_\sigma^n|\mathbf{u}_\sigma^n|^2 = p_0 \sum_{\ell=1}^2 \frac{z_{\ell,\sigma}^n}{\gamma_\ell - 1}+\sum_{\ell=1}^2 \frac{z_{\ell,\sigma}^n \gamma_\ell p_{w,\ell}}{\gamma_\ell - 1},
\end{equation*}
which indicates
\begin{equation*}
    p_\sigma^n = \frac{E_\sigma^n - \frac{1}{2}\rho_\sigma^n|\mathbf{u}_\sigma^n|^2 -\frac{\gamma_\sigma^n p_w}{\gamma_\sigma^n-1}}{\sum_{\ell=1}^2 \frac{z_{\ell,\sigma}^n}{\gamma_\ell - 1}} \equiv p_0.
\end{equation*}    
Thus, uniform pressure is also preserved. This completes the proof. 
\end{proof}

\subsection{Bound-preserving limiter}
For gas-gas/gas-liquid two-phase flows, 
although the oscillation-eliminating procedure effectively 
suppresses numerical oscillations, it does not guarantee that
 \begin{equation}\label{U_ingp}
    \mathbf{U}_\sigma^n(x,y) \in \mathcal G,\quad  \forall (x,y) \in \mathbb{S}_{i,j},
 \end{equation}
 where $\mathcal G=\left\{ \mathbf U| z_1\rho_1 > 0, z_2\rho_2 > 0, \ \tilde{c}^2(\mathbf U)>0,\ z_1\in(0,1) \right\}$
  and $\mathbb{S}_{i,j}$ represents the collection of Gauss-Lobatto quadrature points within the cell $I_{i,j}$.
 To achieve this, 
 we apply the bound-preserving limiter which is introduced in \cite{yan2024uniformly,wang2024bound,cheng2020discontinuous,cheng2022bound,wang2025bound} 
 after the oscillation-eliminating procedure. The computational details are given as follows:

\textbf{Step 1.} Enforce bounded volume fraction $\varepsilon_z \le z_1 \le 1 - \varepsilon_z$ and positive partial densities $z_\ell \rho_\ell \ge \varepsilon_\ell$ by modifying the OEDG solution $\mathbf{U}_\sigma^n(x,y)$ into
\begin{equation}\label{BP_Z}
    \widetilde{\mathbf{U}}_\sigma^n(x,y) = \overline{\mathbf{U}}_{i,j}^n + \theta_{z\rho} \left( \mathbf{U}_\sigma^n(x,y) - \overline{\mathbf{U}}_{i,j}^n \right), \quad \theta_{z\rho} = \min\{ \theta_{z_1\rho_1}, \theta_{z_2\rho_2}, \theta_{z_1} \},
\end{equation}
where $\varepsilon_\ell = \min\{ (\overline{z_\ell \rho_\ell})_{i,j}^n, 10^{-13} \}$, $\varepsilon_z = \min \{ (\overline{z_1})_{i,j}^n,1-(\overline{z_1})_{i,j}^n , 10^{-13} \}$, and
\begin{equation*}
    \theta_{z_\ell \rho_\ell} = \begin{cases}
    \frac{(\overline{z_\ell \rho_\ell})_{i,j}^n - \varepsilon_\ell}{(\overline{z_\ell \rho_\ell})_{i,j}^n - \min\limits_{(x,y) \in \mathbb{S}_{i,j}} (z_\ell \rho_\ell)_\sigma^n(x,y)}, & \text{if } \min\limits_{(x,y) \in \mathbb{S}_{i,j}} (z_\ell \rho_\ell)_\sigma^n(x,y) < \varepsilon_\ell, \\
    1, & \text{otherwise},
    \end{cases}
\end{equation*}
and
\begin{equation*}
    \theta_{z_1} = \begin{cases} 
    \frac{(\overline{z_1})_{i,j}^n - \varepsilon_z}{(\overline{z_1})_{i,j}^n - m_{i,j}}, & \text{if } m_{i,j} < \varepsilon_z, \\
    \frac{(1 - \varepsilon_z) - (\overline{z_1})_{i,j}^n}{M_{i,j} - (\overline{z_1})_{i,j}^n}, & \text{if } M_{i,j} > 1 - \varepsilon_z, \\
    1, & \text{otherwise},
    \end{cases}
\end{equation*}
with $M_{i,j} = \max_{(x,y) \in \mathbb{S}_{i,j}} (z_1)_\sigma^n(x,y)$ and $m_{i,j} = \min_{(x,y) \in \mathbb{S}_{i,j}} (z_1)_\sigma^n(x,y)$.

\noindent \textbf{Step 2.} Enforce $\tilde{c}^2 \ge \varepsilon_c$ by modifying the solution $\widetilde{\mathbf{U}}_\sigma^n(x,y)$ into
\begin{equation}\label{BP_P}
    \mathbf{U}_{i,j}^n(x,y) = \overline{\mathbf{U}}_{i,j}^n + \theta_{c} \left( \widetilde{\mathbf{U}}_\sigma^n(x,y) - \overline{\mathbf{U}}_{i,j}^n \right),
\end{equation}
where $\varepsilon_c = \min\{ \tilde{c}^2(\overline{\mathbf{U}}_{i,j}^n), 10^{-13} \}$ and $\theta_{c}$ is defined by
\begin{equation*}
    \theta_{c} = \begin{cases}
    \frac{\tilde{c}^2(\overline{\mathbf{U}}_{i,j}^n) - \varepsilon_c}{\tilde{c}^2(\overline{\mathbf{U}}_{i,j}^n) - \min\limits_{(x,y) \in \mathbb{S}_{i,j}} \tilde{c}^2(\widetilde{\mathbf{U}}_\sigma^n(x,y))}, & \text{if } \min\limits_{(x,y) \in \mathbb{S}_{i,j}} \tilde{c}^2(\widetilde{\mathbf{U}}_\sigma^n(x,y)) < \varepsilon_c, \\
    1, & \text{otherwise}.
    \end{cases}
\end{equation*}

It is straightforward to verify that the solution after applying the bound-preserving limiter 
satisfies \eqref{U_ingp} and maintains the cell-averaged solution. Furthermore, by defining the bound-preserving limiter on element $I_{i,j}$ as
\begin{equation}
    \Pi_h \left( \mathbf{U}_\sigma^n(x,y) \right) := \mathbf{U}_{i,j}^n(x,y) = \overline{\mathbf{U}}_{i,j}^n + \theta_{c} \theta_{z\rho} \left( \mathbf{U}_\sigma^n(x,y) - \overline{\mathbf{U}}_{i,j}^n \right),
\end{equation}
then we can establish its Abgrall condition-preserving property.

\begin{lemma}\label{line-invar-BP}
If the numerical solution satisfies a linear identity $\mathbf{\Lambda}\mathbf{U}_h(x,y) \equiv \mathbf{c}$ for a given $\mathbf{\Lambda} \in \mathbb{R}^{q \times 6}$, $\mathbf{c} \in \mathbb{R}^{q \times 1}$, and $q \in \mathbb{Z}^+$, then it follows that
\begin{equation}
\mathbf{\Lambda} \Pi_h (\mathbf{U}_h(x,y)) \equiv \mathbf{c},
\end{equation}
which implies the bound-preserving limiter is linearity-invariant.
\end{lemma}

\begin{theorem}\label{BP_Abgrall}
  The bound-preserving limiter \eqref{BP_Z} and \eqref{BP_P} satisfy the Abgrall condition around an isolated material interface, that is
\begin{equation*}
    \text{if } \mathbf{u}_\sigma^n(x,y) \equiv \mathbf{u}_0, \quad p_\sigma^n(x,y) \equiv p_0, \quad \text{then } \mathbf{u}_{i,j}^n(x,y) \equiv \mathbf{u}_0, \quad p_{i,j}^{n+1}(x,y) \equiv p_0,
\end{equation*}
where $\mathbf{u}_{i,j}^n := \mathbf{u}(\Pi_h \mathbf{U}^n_\sigma(x,y))$ and $p_{i,j}^{n+1} := p(\Pi_h \mathbf{U}^n_\sigma(x,y))$.
\end{theorem}
The proofs of Lemma \ref{line-invar-BP} and Theorem \ref{BP_Abgrall}  are similar to that of Lemma \ref{line-invar} and Theorem \ref{OEDG_Abgrall}, respectively, are therefore omitted here.

\subsection{Summary of the numerical algorithm}
In this section, we present the computational procedure flow with limiters in Table \ref{tab:computational_flow}. Based on second-order Strang 
splitting method, we have three steps: the first and third step are implicit solver, 
and the second step is a quasi-conservative DG solver. To suppress numerical oscillations and maintain physical constraint of the solution, we use the OEDG procedure and bound-preserving limiter after each stage of the Runge-Kutta scheme.

\begin{table}[H]
\centering
\caption{Computational Procedure Flow with Limiters.}
\label{tab:computational_flow}
\renewcommand{\arraystretch}{1.5}
\begin{tabular}{lll}
\toprule
\textbf{Computational Steps} &\textbf{Solvers} & \textbf{Limiters} \\
\midrule
1. Strang Splitting Step 1 : &Implicit solver for Eq. \eqref{kappa} $(\Delta t/2)$ &  \textbf{OEDG+BP limiter} \\
\addlinespace
2. Strang Splitting Step 2 : &DG solver for Eq. \eqref{eq:RHS} $(\Delta t)$&  
   \textbf{OEDG+BP limiter} \\
\addlinespace
3. Strang Splitting Step 3 : &Implicit solver for Eq. \eqref{kappa} $(\Delta t/2)$ &  \textbf{OEDG+BP limiter} \\
\bottomrule
\end{tabular}
\end{table}

\section{Numerical experiments}\label{sec-experiments}
In this section, 
a series of one- and two-dimensional benchmark cases are presented to 
validate the proposed BP-OEDG method coupled with the 
operator splitting method. A uniform time step scaling of $\Delta t \propto \Delta x$ with a CFL 
number of $0.1$ is applied across all simulations. 
While the operator splitting strategy inherently 
degrades the interfacial temporal accuracy 
to second order under this scaling, 
this level of accuracy is practically 
adequate for capturing contact discontinuities in two-phase flow dynamics.

\subsection{One-dimensional smooth advaction problem}\label{sec-gas_gas_smooth_cases}
The one-dimensional smooth test case \cite{yan2024uniformly} is 
simulated to verify the accuracy of the proposed method. 
The computational domain is defined as $[0, 2]$ with 
periodic boundary conditions, and the initial conditions are prescribed as follows:
\[
(\rho_1, \rho_2, u, p, z_1) = \bigl(1,\; 1.5 + 0.4\cos(\pi x),\; 1 + 0.4\cos(\pi x),\; 1,\; 0.5 + 0.4\sin(\pi x) \bigr),
\]
where the parameters are \(\gamma_1 = 1.4, p_{w,1} = 0\) and \(\gamma_2 = 1.6, p_{w,2} = 0\). 
The problem is simulated until \(t = 0.05\). 
Since the velocity field is spatially varying, 
the non-conservative term $\kappa  \nabla \cdot \mathbf{u}$ cannot be omitted.

The $L^1$, $L^2$, and $L^\infty$ errors 
for the volume fraction $z_1$, along with the empirical convergence rates, 
are presented in Table \ref{tab:gas-gas-error_convergence}. Here, the reference solution is generated by 
solving Kapila's model \cite{zhang2023analysis} 
using a highly refined BP-OEDG($P^2$) discretization with 
30,000 uniform cells.
The results demonstrate that the proposed method 
successfully achieves the expected convergence rates. 
Notably, while the Strang splitting and the implicit solver formally 
restrict the scheme to second-order temporal accuracy, 
the temporal error remains subdominant to the 
spatial discretization error. 
Consequently, the DG($P^2$) configuration 
with a $\Delta t \propto \Delta x$ scaling still 
manifests optimal third-order convergence.
\begin{table}[H]
  \centering
  \caption{Numerical errors and convergence orders for Case \ref{sec-gas_gas_smooth_cases}.}
  \label{tab:gas-gas-error_convergence}
  \begin{tabular}{ccccccc}
    \toprule
    $N_x$ & $L^1$ Error & Order & $L^2$ Error & Order & $L^\infty$ Error & Order \\
    \midrule
    \multicolumn{7}{c}{BP-OEDG($P^1$) method}  \\
    20  & 1.52E-3 & -    & 1.74E-3 & -    & 5.37E-3 & -    \\
    40  & 3.68E-4 & 2.05 & 4.34E-4 & 2.01 & 1.35E-3 & 1.99 \\
    80  & 9.14E-5 & 2.01 & 1.08E-4 & 2.00 & 3.42E-4 & 1.99 \\
    160 & 2.28E-5 & 2.00 & 2.71E-5 & 2.00 & 8.56E-5 & 2.00 \\
    320 & 5.71E-6 & 2.00 & 6.77E-6 & 2.00 & 2.14E-5 & 2.00 \\
    \midrule
\multicolumn{7}{c}{BP-OEDG($P^2$) method} \\
20  & 9.92E-5 & -    & 9.53E-5 & -    & 2.90E-4 & -    \\
    40  & 1.30E-5 & 2.93 & 1.38E-5 & 2.79 & 4.55E-5 & 2.67 \\
    80  & 1.67E-6 & 2.96 & 1.81E-6 & 2.93 & 6.45E-6 & 2.82 \\
    160 & 2.10E-7 & 2.99 & 2.28E-7 & 2.99 & 8.23E-7 & 2.97 \\
    320 & 2.64E-8 & 2.99 & 2.86E-8 & 2.99 & 1.04E-7 & 2.98 \\
    \bottomrule
  \end{tabular}
\end{table}

\subsection{One-dimensional gas-liquid isolated interface}\label{isolated-case}
This case is widely used in the literature 
\cite{yan2024uniformly,cheng2020quasi,pandare2018robust,johnsen2006implementation} to 
verify the Abgrall condition across an isolated material interface. 
The initial conditions are given as follows
\[
(\rho_1, \rho_2, u, p, z_1) = 
\begin{cases}
(1.0, 1000.0, 2.0, 1.0, 1.0 - 10^{-10}), & -5 < x \leq 0, \\
(1.0, 1000.0, 2.0, 1.0, 10^{-10}), & 0 < x < 5,
\end{cases}
\]
where the parameters are set to
$\gamma_1 = 1.4$, $p_{w, 1} = 0.0$ and $\gamma_2 = 4.4$, $p_{w, 2} = 6000$. 
The computational domain is defined as $[-5,5]$. 
The simulations are performed using the BP-OEDG($P^2$) method with 
100, 200, and 400 elements, up to a final time of $t=1.0$. Figure 
\ref{fig:isolated_distributions} compares the numerical 
solutions for density, pressure, velocity, and 
volume fraction against the reference solutions. 
As the mesh is refined, the numerical solutions consistently 
converge to the reference solutions. Notably, 
the results also confirm that the proposed method successfully 
preserves the Abgrall condition at the material interface.
\begin{figure}\label{fig:isolated_distributions}
  \centering
  \begin{minipage}[b]{0.48\textwidth}
    \centering
    \includegraphics[width=\linewidth]{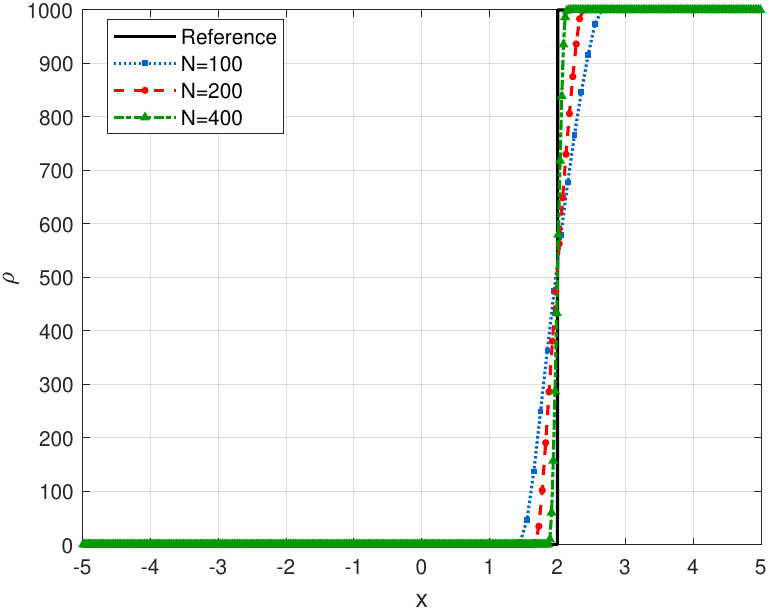}
  \end{minipage}
  \hfill
  \begin{minipage}[b]{0.48\textwidth}
    \centering
    \includegraphics[width=\linewidth]{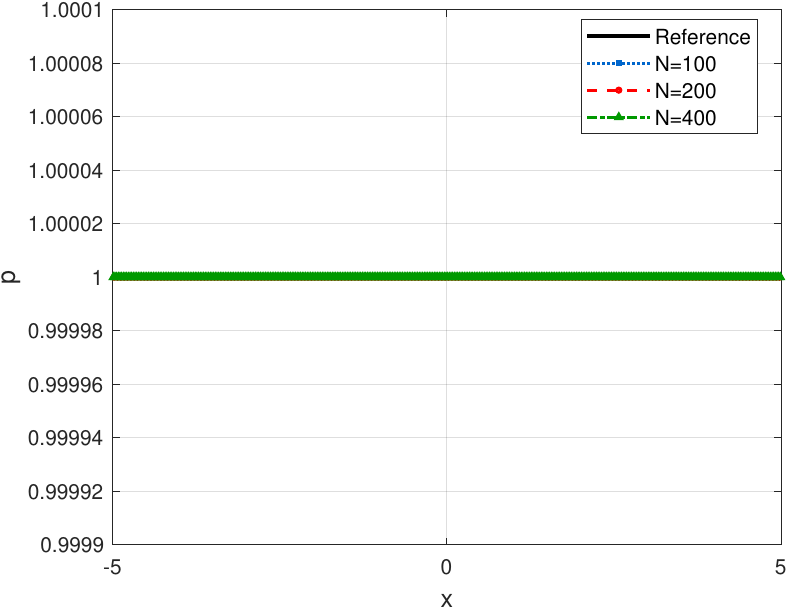}
  \end{minipage}
  
  \vspace{1em} 
  
  \begin{minipage}[b]{0.48\textwidth}
    \centering
    \includegraphics[width=\linewidth]{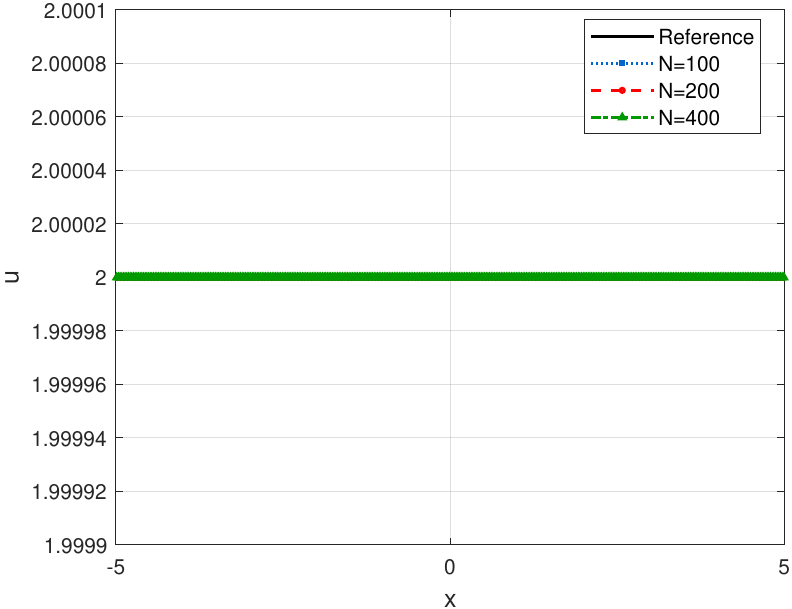}
  \end{minipage}
  \hfill
  \begin{minipage}[b]{0.48\textwidth}
    \centering
    \includegraphics[width=\linewidth]{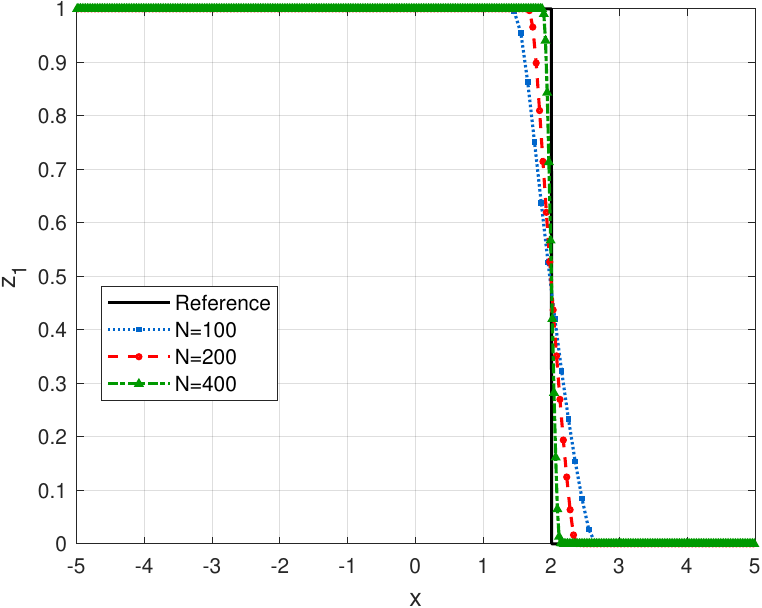}
  \end{minipage}
  
  \caption{Numerical results computed by the DG ($P^2$) in Case \ref{isolated-case}. 
  Top left: mixture density $\rho$. Top right: pressure $p$. Bottom left: velocity $u$. Bottom right: volume fraction $z_1$.}
  \label{fig:isolated_distributions}
\end{figure}

\subsection{One-dimensional double rarefaction problem}\label{doublewave-case}
This test case \cite{cheng2020quasi,zhang2023analysis,yan2024uniformly,zhang2010positivity} is designed to verify 
the bound-preserving property of the BP-OEDG method when 
coupled with the operator splitting method. 
The initial conditions are prescribed as follows
\[
(\rho_1, \rho_2, u, p, z_1) = 
\begin{cases}
(2.0, 2.0, -1.0, 0.2, 1.0 - 10^{-6}), & -1 < x \leq 0, \\
(2.0, 2.0, 1.0, 0.2, 10^{-6}), & 0 < x < 1,
\end{cases}
\]
where the parameters are set to $\gamma_1 = 1.4$, $p_{w, 1} = 0.0$ 
and $\gamma_2 = 4.4$, $p_{w, 2} = 0.0$.
The computational 
domain is defined as $[-1, 1]$ and is 
discretized using a uniform mesh of 2000 elements. 
The simulation is advanced to a final 
time of $t = 0.4$. The BP-OEDG($P^2$) scheme 
is employed to solve the Kapila five-equation model (denoted as ``BP-OEDG"). 
For comparison, a quasi-conservative DG($P^2$) scheme is 
applied to the five-equation transport model on 
the identical mesh (denoted as ``QC"). 

Figure~\ref{fig:doublewave_distributions} presents 
the computed profiles for 
pressure, internal energy, velocity, and 
mixture density. It is evident that both 
the 'BP-OEDG' and 'QC' numerical 
solutions show excellent agreement with the reference solution. 
Notably, the results illustrated in Figure~\ref{fig:doublewave_z1} reveal 
a significant discrepancy 
in the distribution of the volume 
fraction $z_1$ between the two models. 
This deviation is primarily attributed 
to the distinct compressibility of the two fluid phases. 
\begin{figure}[H]
  \centering
  \begin{minipage}[b]{0.48\textwidth}
    \centering
    \includegraphics[width=\linewidth]{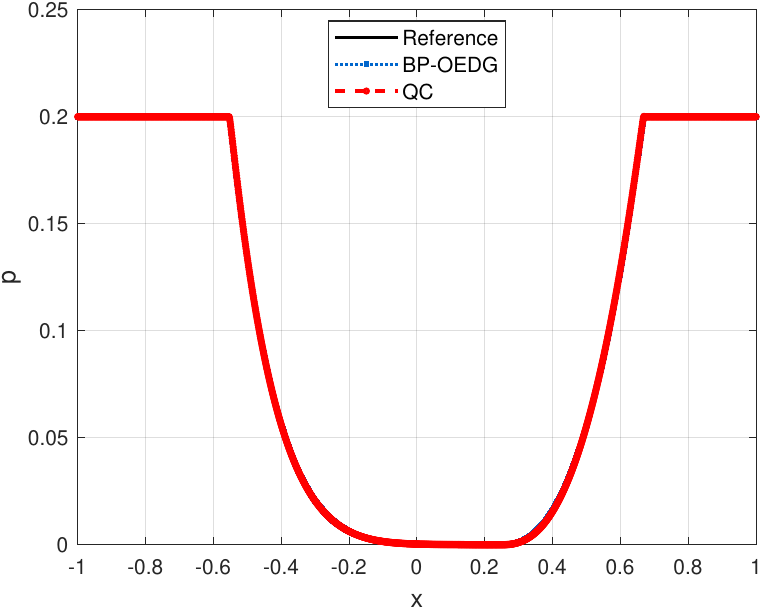}
  \end{minipage}
  \hfill
  \begin{minipage}[b]{0.48\textwidth}
    \centering
    \includegraphics[width=\linewidth]{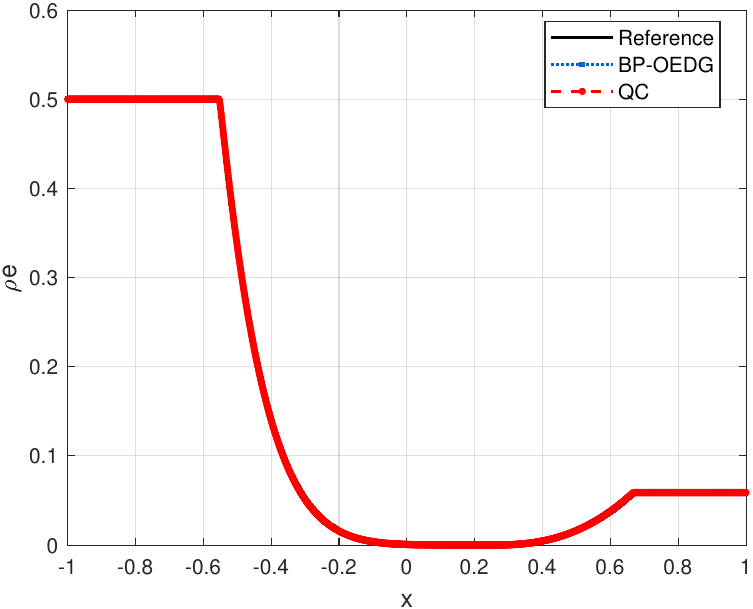}
  \end{minipage}

  \vskip\baselineskip 

  \begin{minipage}[b]{0.48\textwidth}
    \centering
    \includegraphics[width=\linewidth]{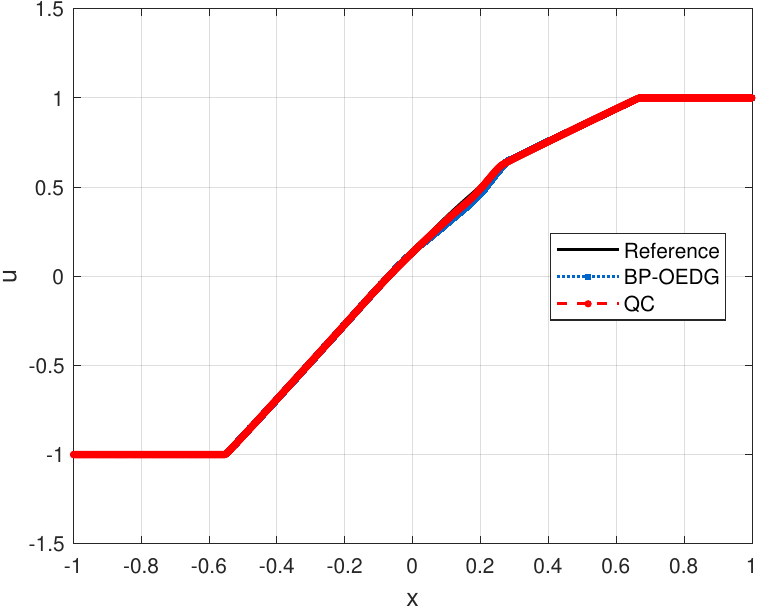}
  \end{minipage}
  \hfill
  \begin{minipage}[b]{0.48\textwidth}
    \centering
    \includegraphics[width=\linewidth]{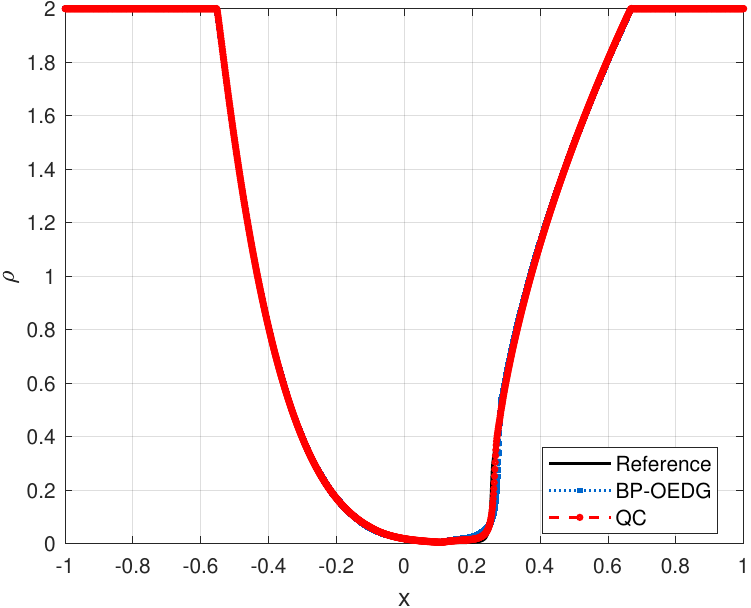}    
  \end{minipage}

  \caption{Numerical results computed by the DG ($P^2$) in case \ref{doublewave-case}. 
  Top left: pressure $p$. Top right: internal energy $\rho e$.
  Bottom left: velocity $u$. Bottom right: mixture density $\rho$.}
  \label{fig:doublewave_distributions}
\end{figure}

\begin{figure}[H]
  \centering
  \includegraphics[width=0.5\textwidth]{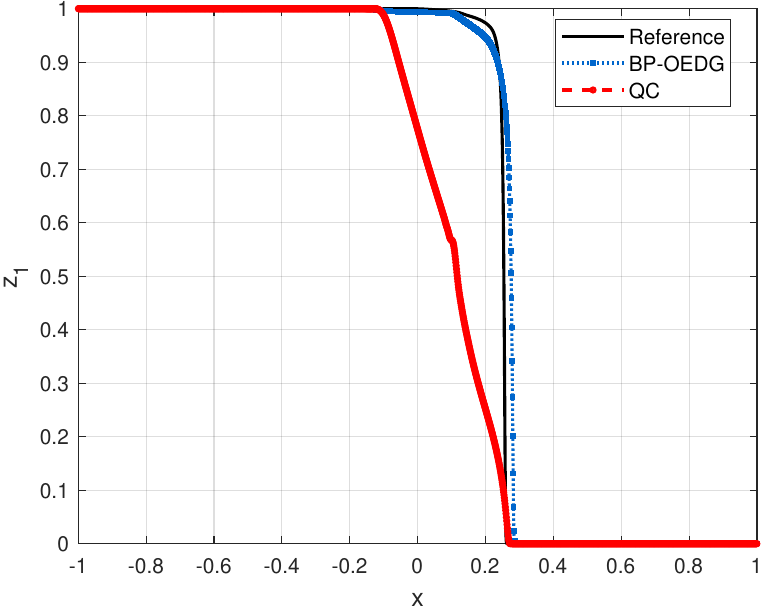}
  \caption{Volume fraction $z_1$ in Case \ref{doublewave-case}.}
  \label{fig:doublewave_z1}
\end{figure}

\subsection{One-dimensional gas-liquid Riemman problem}\label{gasliquid-case}
This benchmark considers 
a gas-liquid Riemann problem, 
akin to the configurations extensively 
investigated in the literature \cite{luo2004computation,xu2017explicit,liu2003ghost}. 
The initial conditions are prescribed as follows:
\[
(\rho_1, \rho_2, u, p, z_1) = 
\begin{cases}
(1.27, 1.0, 0.0, 8000.0, 1.0 - 10^{-10}), & -5 < x \leq 0, \\
(1.27, 1.0, 0.0, 1.0, 10^{-10}), & 0 < x < 5,
\end{cases}
\]
where the parameters are $\gamma_1 = 1.4$, $p_{w,1} = 0.0$ 
and $\gamma_2 = 7.15$, $p_{w,2} = 3309$. 
The computational domain is defined as $[-5, 5]$, 
and the simulation is advanced to a 
final time of $t = 0.015$. The computations are performed 
using the BP-OEDG($P^2$) method on a uniform mesh consisting of 800 elements. 

The computed profiles of pressure, 
mixture density, energy density, and 
volume fraction are compared against the 
reference solutions in Figure~\ref{fig:gas_liquid_riemann_distributions}. 
The numerical results exhibit 
excellent agreement with the reference solutions, 
effectively demonstrating the high-order 
accuracy and robustness of the proposed BP-OEDG method.

\begin{figure}\label{fig:gas_liquid_riemann_distributions}
  \centering
  \begin{minipage}[b]{0.48\textwidth}
    \centering
    \includegraphics[width=\linewidth]{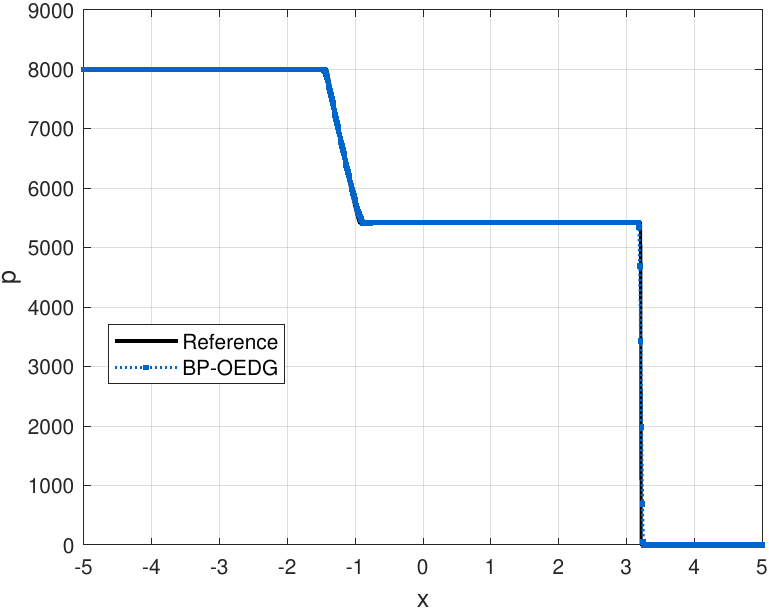}
  \end{minipage}
  \hfill
  \begin{minipage}[b]{0.48\textwidth}
    \centering
    \includegraphics[width=\linewidth]{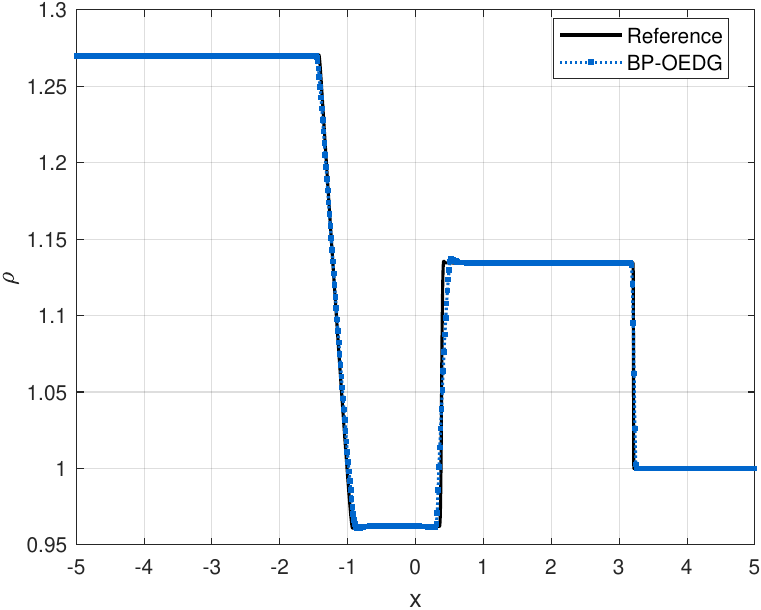}
  \end{minipage}

  \vspace{1em}

  \begin{minipage}[b]{0.48\textwidth}
    \centering
    \includegraphics[width=\linewidth]{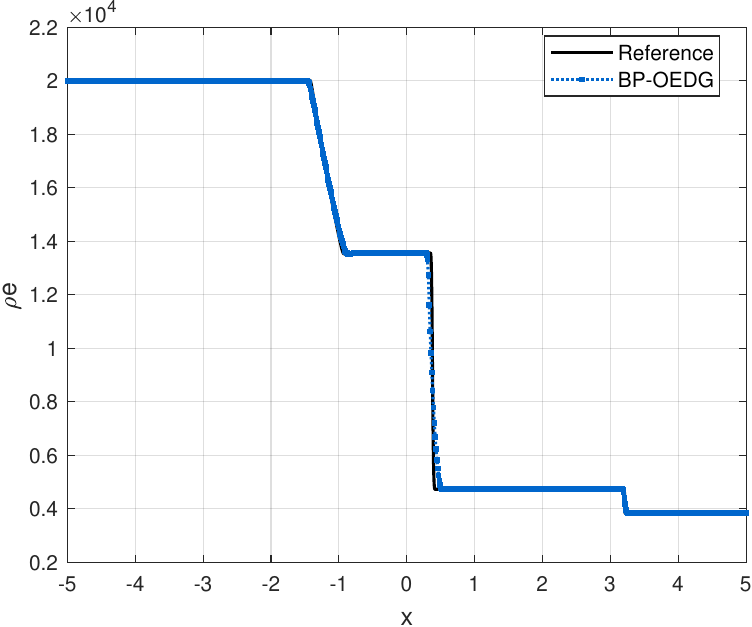}   
  \end{minipage}
  \hfill
  \begin{minipage}[b]{0.48\textwidth}
    \centering
    \includegraphics[width=\linewidth]{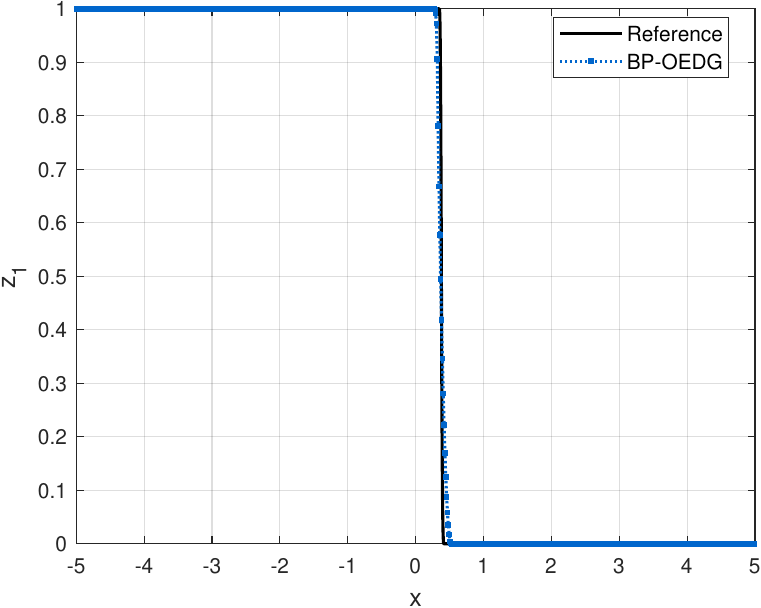}
  \end{minipage}
  
  \caption{Numerical results computed by the BP-OEDG ($P^2$) in Case \ref{gasliquid-case}. 
  Top left: pressure $p$. Top right: mixture density $\rho$. 
  Bottom left: energy density $\rho e$. Bottom right: volume fraction $z_1$.}
  \label{fig:gas_liquid_riemann_distributions}
\end{figure}

\subsection{One-dimensional gas-liquid shock tube} 
This benchmark \cite{cheng2020quasi,saurel1999simple,cheng2014positivity} considers a one-dimensional
gas-liquid shock tube problem characterized by 
exceptionally high density and pressure ratios. 
The initial conditions are prescribed as follows
\[
(\rho_1, \rho_2, u, p, z_1) = 
\begin{cases}
(1.0, 200.0, 0.0, 10^5, 1.0 - 10^{-6}), & -1 < x \leq 0, \\
(1.0, 200.0, 0.0, 10^9, 10^{-6}), & 0 < x < 1,
\end{cases}
\]
where the parameters are $\gamma_1 = 1.4$, $p_{w,1} = 0.0$ and $\gamma_2 
= 4.4$, $p_{w,2} = 6000$. 
The computational domain is 
defined as $[-1, 1]$ and is discretized 
using a uniform mesh consisting of 2000 elements. 
The computations are performed using the BP-OEDG($P^2$) scheme 
up to a final simulation time of $t = 0.0002$. 
The exact solution to this 
problem develops a left-moving 
shock wave, a right-moving rarefaction wave, 
and a material interface (contact discontinuity) 
propagating in between. 
The computed profiles for the 
logarithm of pressure, mixture density, 
velocity, and volume fraction are compared 
against the reference exact solutions 
in Figure~\ref{fig:shocktube_distributions}. 
It is evident that the proposed BP-OEDG method 
accurately resolves the shock, 
the contact discontinuity, and the rarefaction wave, 
exhibiting excellent agreement with the reference solutions.

\begin{figure}[H]\label{fig:shocktube_distributions}
  \centering
  \begin{minipage}[b]{0.48\textwidth}
    \centering
    \includegraphics[width=\linewidth]{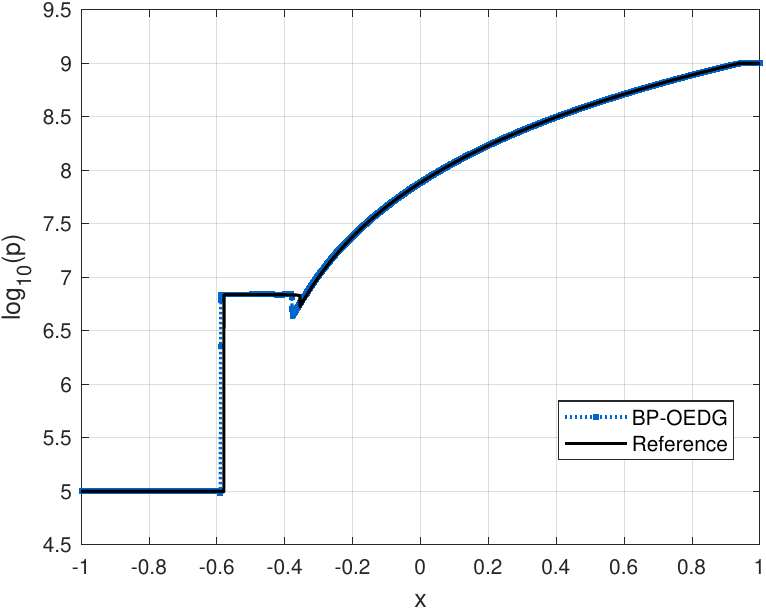}
  \end{minipage}
  \hfill
  \begin{minipage}[b]{0.48\textwidth}
    \centering
    \includegraphics[width=\linewidth]{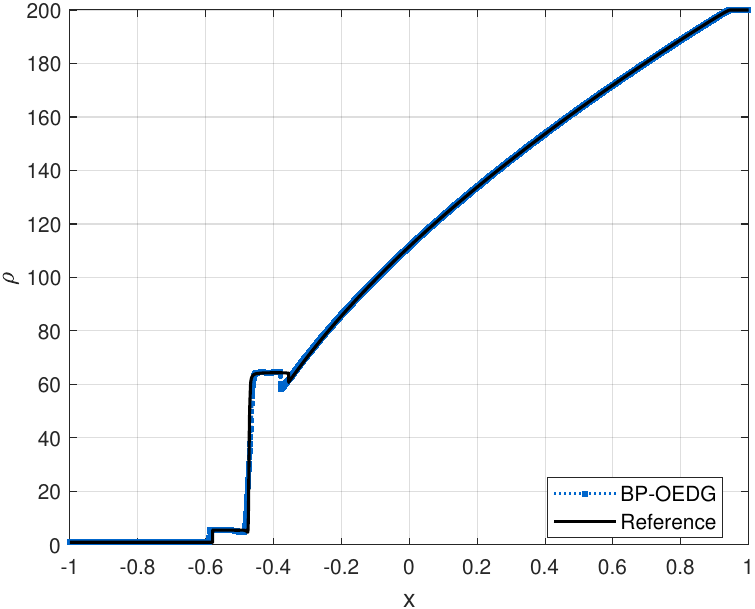}
  \end{minipage}

  \vspace{1em}

  \begin{minipage}[b]{0.48\textwidth}
    \centering
    \includegraphics[width=\linewidth]{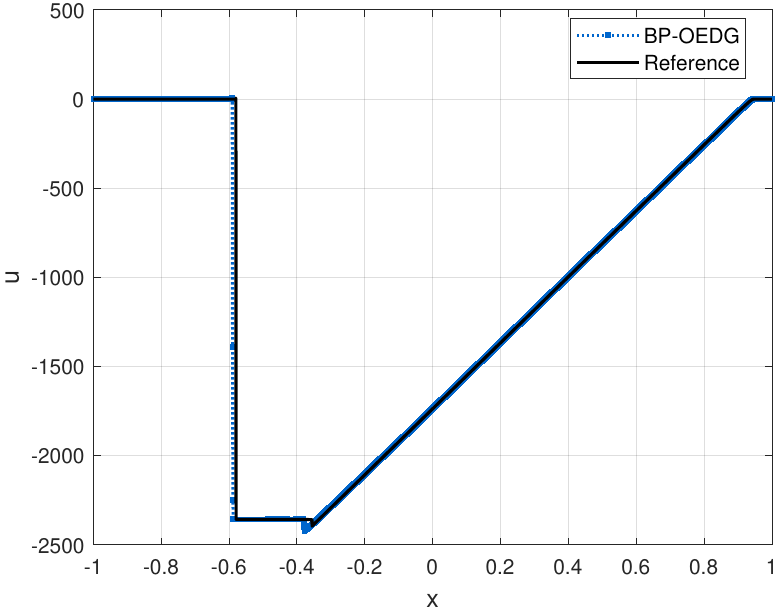}
  \end{minipage}
  \hfill
  \begin{minipage}[b]{0.48\textwidth}
    \centering
    \includegraphics[width=\linewidth]{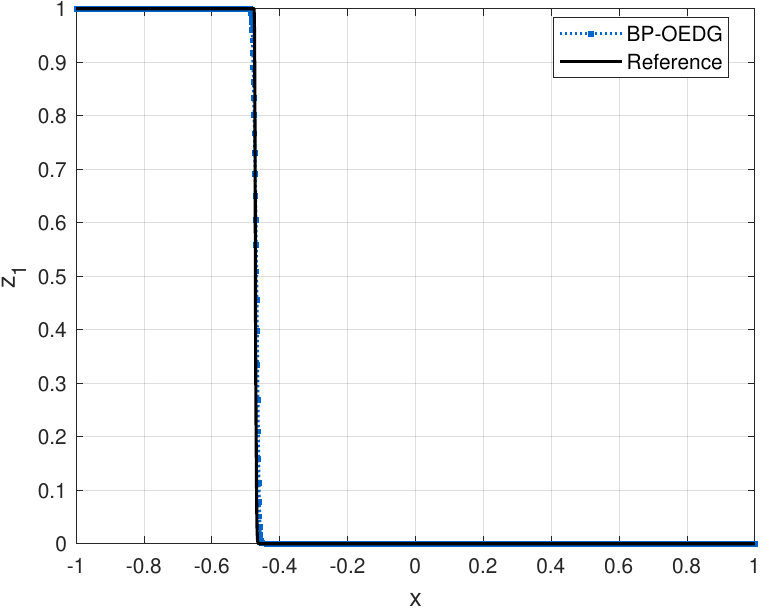}
  \end{minipage}
  
\caption{Numerical results computed by the DG ($P^2$) in Case \ref{doublewave-case}. 
Top left: logarithm of pressure $\log_{10}(p)$. Top right: mixture density $\rho$. 
Bottom left: velocity $u$. Bottom right: volume fraction $z_1$.}
  \label{fig:shocktube_distributions}
\end{figure}

\subsection{Two-dimensional isentropic vortex problem}\label{vortex-case}
To verify the high-order accuracy away from the material interfaces, we 
consider the isentropic vortex problem \cite{zhang2023analysis,white2025high,
zhang2012positivity,cai2016positivity}, which is designed for single-fluid flows. 
The uniform flow field is given by values
 $\rho_{\infty} = p_{\infty} = u_{\infty} = = v_{\infty} = 1$ and
  $T_{\infty} = p_{\infty}/ \rho_{\infty}$. 
  A vortex perturbation centered at $(x_c,y_c)$ is added to the uniform flow 

\begin{equation*}
[\delta u, \delta v]^{\top} = \frac{\epsilon}{2\pi} e^{0.5(1-r^2)} [-(y-y_c), (x-x_c)]^{\top}, \quad \delta T = -\frac{(\gamma - 1)\epsilon^2}{8\gamma\pi^2} e^{(1-r^2)},
\end{equation*}
where $r = \sqrt{(x-x_c)^2 + (y-y_c)^2}$ is the distance from the vortex center, 
and $\epsilon = 10.0828$ is the perturbation intensity. To satisfy the 
single-fluid configuration,
we set $z_1 = 0.5, \gamma_1 = \gamma_2 = \gamma =1.4$. 
Since the flow field is isentropic, we have 
\begin{equation*}
\rho = (T_{\infty} + \delta T)^{\frac{1}{\gamma-1}}, \quad u = u_{\infty} + \delta u, \quad v = v_{\infty} + \delta v, \quad p = \rho^{\gamma}.
\end{equation*}

The computation domain is \([-5, 5]\times[-5,5]\) with periodic boundary conditions, the 
vortex is initialized at $(x_c,y_c)=(0,0)$, and the output time is $T = 0.05$. 
Under the initial conditions, the lowest density of the exact solution are $7.8 \times 10^{-15}$ \cite{zhang2012positivity,cai2016positivity}, and therefore the bound-preserving limiter is required to prevent the emergence of non-physical values.
Ideally, the vortex profile should 
undergo pure advection without 
any distortion over time, 
meaning that the exact solution at time $t$ is simply a 
rigid translation of the initial condition by $t(u_{\infty}, v_{\infty})$, 
as schematically depicted in Figure~\ref{fig:isentropic_vortex_setup}. 
The numerical results are listed in Table~\ref{tab:isentropic_vortex}, 
which include the convergence order of density for the BP-OEDG method, and the lowest density at final time during the computation. 

\begin{figure}[htbp]
    \centering
    \begin{tikzpicture}[scale=0.5, >=stealth]

        \def\side{8}       
        \def\half{4}       
        
        \def\vorZX{4}      
        \def\vorZY{4}
        
        \def\vorTX{5.5}    
        \def\vorTY{5.5}

        \fill[blue!5] (0,0) rectangle (\side, \side);

        \draw[thick] (0,0) rectangle (\side, \side);

        \draw[thick, fill=white] (\vorZX, \vorZY) circle (1.2);
        \draw[dashed, gray] (\vorZX, \vorZY) circle (0.8);
        \draw[dashed, gray] (\vorZX, \vorZY) circle (0.4);
        \fill (\vorZX, \vorZY) circle (0.08); 
        
        \draw[->, semithick, black!70] (\vorZX+0.8, \vorZY) arc [start angle=0, end angle=90, radius=0.8];
        \draw[->, semithick, black!70] (\vorZX-0.8, \vorZY) arc [start angle=180, end angle=270, radius=0.8];

        \draw[thick, dashed, fill=white, fill opacity=0.6] (\vorTX, \vorTY) circle (1.2);
        \draw[dotted, gray] (\vorTX, \vorTY) circle (0.8);
        \fill [gray] (\vorTX, \vorTY) circle (0.08);

        \draw[->, ultra thick, red] (\vorZX, \vorZY) -- (\vorTX, \vorTY);

        \node[below, scale=0.8, font=\bfseries] at (\vorZX, \vorZY-1.3) {Initial Vortex $(t=0)$};
        \node[above right, scale=0.8, font=\bfseries, text=black!60] at (\vorTX-5.0, \vorTY+1.5) {Final Vortex $(t>0)$};
        
    \end{tikzpicture}
    \caption{Schematic layout of the isentropic vortex transport problem.}
    \label{fig:isentropic_vortex_setup}
\end{figure}

\begin{figure}[htbp]
    \centering
    \begin{tikzpicture}[scale=25, >=stealth] 

        \def\xa{-0.3}
        \def\xb{0.15}
        \def\ya{-0.0445}
        \def\yb{0.0445}
        \def\xs{0.05}     
        \def\rb{0.025}    

        \fill[gray!10] (\xa, \ya) rectangle (\xs, \yb);
        \fill[gray!30] (\xs, \ya) rectangle (\xb, \yb);
        \fill[blue!15, draw=blue, thick] (0,0) circle (\rb);

        \node[blue!80!black, font=\bfseries] at (0,0) {He};

        \draw[thick] (\xa, \ya) rectangle (\xb, \yb);

        \draw[red, ultra thick] (\xs, \ya) -- (\xs, \yb);
        
        \draw[->, ultra thick, red] (\xs+0.03, 0) -- (\xs-0.01, 0);

        \node[scale=0.7] at (-0.15, 0.025) {Pre-shock Air};
        \node[scale=0.7] at (0.1, 0.025) {Post-shock Air};

        \node[below, scale=0.8] at (\xa, \ya) {$\xa$};
        \node[below, scale=0.8] at (\xb, \ya) {$\xb$};
        \node[below, scale=0.8] at (0, \ya) {$0$};
        \node[below, scale=0.8, red] at (\xs, \ya) {$\xs$}; 
        
        \foreach \y in {\ya, 0, \yb} {
            \draw (\xa, \y) -- (\xa-0.005, \y); 
        }
        \node[left, scale=0.8] at (\xa-0.005, \ya) {$\ya$};
        \node[left, scale=0.8] at (\xa-0.005, \yb) {$\yb$};
        \node[left, scale=0.8] at (\xa-0.005, 0) {$0$};

    \end{tikzpicture}
    \caption{Initial configuration for shock-helium bubble interaction.}
    \label{fig:shock_bubble_init_he}
\end{figure}

\begin{table}[H]
    \centering
    \caption{Error and convergence order of density for Case \ref{vortex-case}.}
    \label{tab:isentropic_vortex}
    \begin{tabular}{lcccccc}
        \toprule
        \multicolumn{6}{c}{BP-OEDG($P^1$) method} \\
        \midrule
        $N_x \times N_y$ & $L^2$ Error & Order & $L^\infty$ Error & Order & Minimum $\rho$  \\
        \midrule
        $40 \times 40$   & 1.89E-2 &    -  & 4.53E-2 &   -   & 4.40E-4 \\
        $80 \times 80$   & 4.41E-3 & 2.10 & 7.55E-3 & 2.58 & 2.70E-5 \\
        $160 \times 160$ & 1.09E-3 & 2.02 & 1.74E-3 & 2.12 & 1.36E-6 \\
        $320 \times 320$ & 2.74E-4 & 1.99 & 4.35E-4 & 2.00 & 2.00E-7 \\       

        \midrule
        \multicolumn{6}{c}{BP-OEDG($P^2$) method} \\
        $40 \times 40$   & 1.20E-3 & -    & 3.70E-3 & -    & 9.40E-4 \\
        $80 \times 80$   & 1.58E-4 & 2.93 & 4.72E-4 & 2.97 & 4.20E-5 \\
        $160 \times 160$ & 2.37E-5 & 2.74 & 8.83E-5 & 2.42 & 1.76E-6 \\
        $320 \times 320$ & 3.57E-6 & 2.73 & 1.42E-5 & 2.64 & 7.70E-8 \\
        \bottomrule
    \end{tabular}
\end{table}

\subsection{Two-dimensional air shock-helium bubble interaction problem}
In this case, 
we simulate the classic shock-bubble interaction 
problem originally conducted by Hass and 
Sturtevant \cite{haas1987interaction}, 
where a Mach 1.22 planar shock wave impinges upon a helium bubble. 
This configuration has been 
widely adopted as a standard benchmark \cite{yan2024uniformly,zhang2023analysis} 
to evaluate the capability of numerical methods in resolving complex 
wave interactions and interfacial instabilities. 
The computational domain and the initial locations 
of the bubble and the incident shock 
wave are depicted in Figure~\ref{fig:shock_bubble_init_he}. 
The initial conditions are prescribed as follows
\begin{align*}
&(\rho_1, \rho_2, u, v, p, z_1) 
\\
&=\begin{cases} 
(1.4, 0.25463, 0, 0, 10^{5}, 10^{-6}), & \sqrt{x^2 + y^2} \le 0.025, \\
(1.92691, 0.25463, -113.5, 0, 1.5698 \times 10^{5}, 1 - 10^{-6}), & x \ge 0.05, \\
(1.4, 0.25463, 0, 0, 10^{5}, 1 - 10^{-6}), & \text{otherwise},
\end{cases}
\end{align*}
where the parameters 
are given by $\gamma_{1}=1.4$, $p_{w,1}=0.0$ and $\gamma_{2}=1.648$, $p_{w,2}=0.0$. 
The computational domain is discretized 
using a mesh consisting of $1800 \times 356$ elements. 
Regarding the boundary conditions, 
solid wall conditions 
are enforced on the upper and lower boundaries. 
An outflow boundary condition is applied at the left boundary, 
while the exact post-shock state is constantly prescribed at the right boundary.

Figure~\ref{fig:shockair_hithe_comparison} 
presents a comparison between the 
computed density contours using the BP-OEDG($P^2$) scheme 
and the experimental results at $t = 62$, $245$, $427$, and $983~\mu\text{s}$ 
(measured relative to the instant when the 
incident shock first impinges upon 
the bubble \cite{yan2024uniformly,shyue2014eulerian}). 
It is evident that 
the morphological evolution of the helium bubble agrees 
well with the experimental observations at each corresponding instant. 
Furthermore, the shock waves and 
contact discontinuities are accurately resolved 
by the BP-OEDG($P^2$) scheme without any spurious oscillations.
\begin{figure}[H]
  \centering
  \begin{minipage}[b]{0.8\textwidth}
    \centering
    \includegraphics[width=\linewidth]{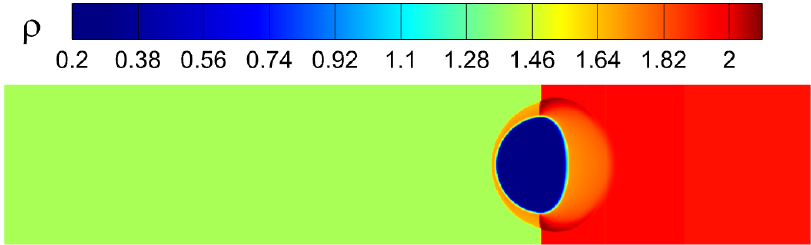}
  \end{minipage}
  \hfill
  \begin{minipage}[b]{0.16\textwidth}
    \centering
     \raisebox{3.6mm}{\includegraphics[width=\linewidth]{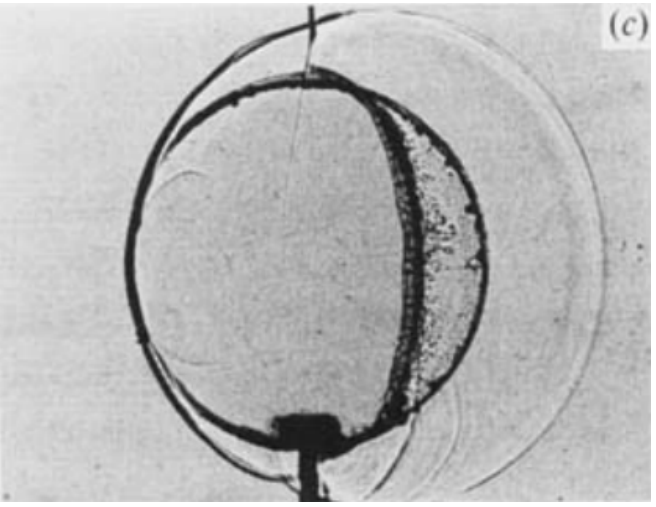}}
  \end{minipage}
  \caption*{$t = 62$ $\mu$ s.}
  \par\medskip
  
  \begin{minipage}[b]{0.8\textwidth}
    \centering
    \includegraphics[width=\linewidth]{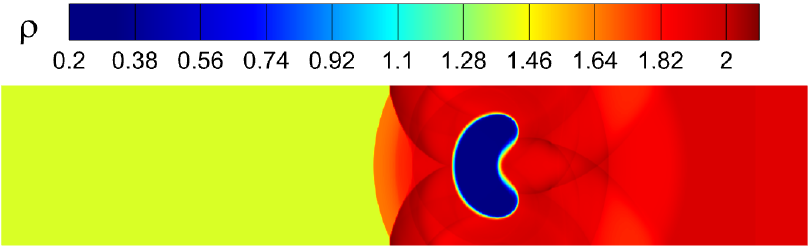}
  \end{minipage}
  \hfill
  \begin{minipage}[b]{0.16\textwidth}
    \centering
    \raisebox{3.6mm}{\includegraphics[width=\linewidth]{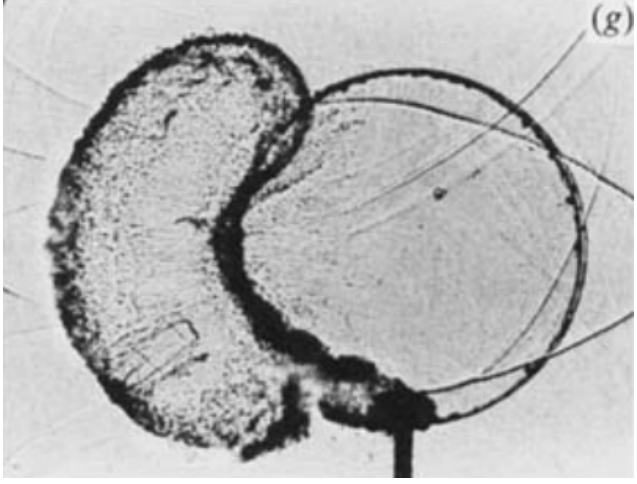}}
  \end{minipage}
  \caption*{$t = 245$ $\mu$ s.}
  \par\medskip
  
  \begin{minipage}[b]{0.8\textwidth}
    \centering
    \includegraphics[width=\linewidth]{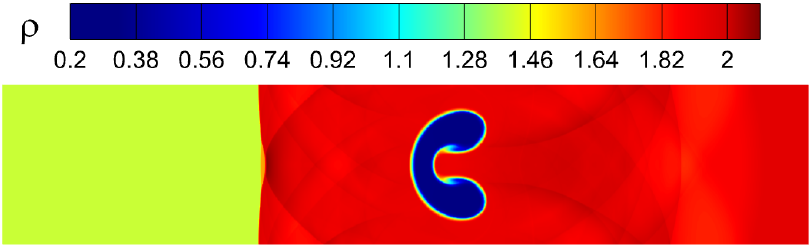}
  \end{minipage}
  \hfill
  \begin{minipage}[b]{0.16\textwidth}
    \centering
    \raisebox{3.6mm}{\includegraphics[width=\linewidth]{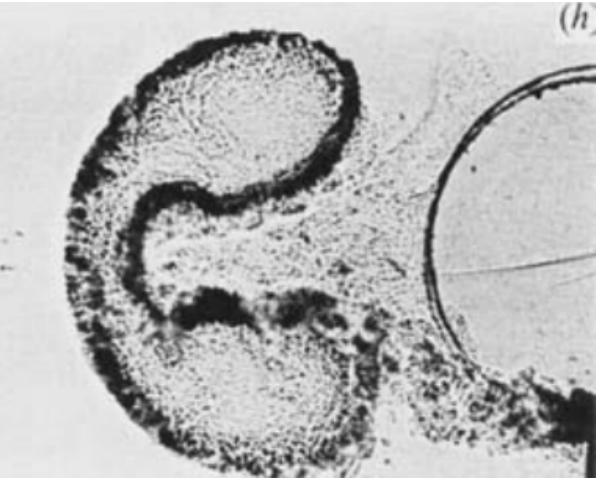}}
  \end{minipage}
  \caption*{$t = 427$ $\mu$ s.}
  \par\medskip
  
  \begin{minipage}[b]{0.8\textwidth}
    \centering
    \includegraphics[width=\linewidth]{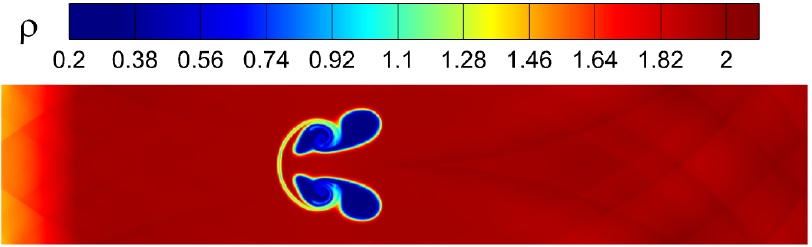}
  \end{minipage}
  \hfill
  \begin{minipage}[b]{0.16\textwidth}
    \centering
    \raisebox{3.6mm}{\includegraphics[width=\linewidth]{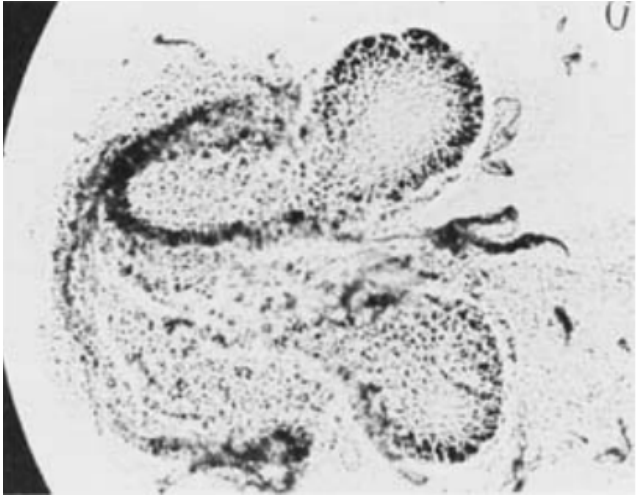}}
  \end{minipage}
  \caption*{$t = 983$ $\mu$ s.}
  
  \caption{Air shock hitting helium bubble: BP-OEDG ($P^2$) vs experiment.}
  \label{fig:shockair_hithe_comparison}
\end{figure}

\subsection{Two-dimensional underwater explosion problem}
In this case, the explosion of a
highly compressed cylindrical air bubble in water under a free surface 
is investigated. This problem can be used to verify the rubustness of the BP-OEDG method. 
This problem has been widely studied in the 
previous literature \cite{zhang2023analysis,white2025high,cheng2020quasi,shukla2010interface}.
The primary computational challenge 
of this benchmark lies in capturing 
the topological changes of the air-water 
interface under extreme density and pressure ratios. 
Although shock reflection at 
the free surface potentially triggers 
cavitation, the present 
study neglects cavitation effects to 
focus strictly on the interface tracking and 
high-order resolution of multi-phase dynamics. 
The initial conditions are as follows:
\begin{equation*}
(\rho_1, \rho_2, u, v, p, z_1) = 
\begin{cases}
(0.001225, 1.0, 0, 0, 1.01325, 10^{-6}), & \quad \text{for water}, \\
(0.001225, 1.0, 0, 0, 1.01325, 1.0 - 10^{-6}), & \quad \text{for air}, \\
(1.25, 1.0, 0, 0, 10000, 1.0 - 10^{-6}), & \quad \text{for air bubble},
\end{cases}
\end{equation*}
where the parameters are $\gamma_1 = 1.4$, $p_{w,1} = 0.0$ 
and $\gamma_2 = 4.4$, $p_{w,2} = 6000$. 
The test is investigated within 
a rectangular domain $ [-2, 2] \times [-1.5, 1.5]$ with wall boundary conditions on all sides, 
which is spatially discretized by 
a uniform grid composed of $600 \times 450$ rectangular cells. 
Initially, the air-water interface 
is aligned with the horizontal line $y = 0$. 
An unperturbed air bubble 
is embedded in the liquid phase, 
centered at $(x_c, y_c) = (0.0, -0.3)$ with a 
radius of $r = 0.12$, 
as schematically illustrated in Fig.~\ref{fig:bubble_interface_initial}. 
Figure~\ref{fig:underwater_explosion} displays
the computed pressure and volume fraction contours at different times.
The primary flow features 
are accurately resolved and show 
excellent agreement with the numerical results 
reported in \cite{cheng2020quasi,white2025high}, 
demonstrating the robustness of 
the BP-OEDG method in handling extreme density and pressure ratios.

\begin{figure}[htbp]
    \centering
    \begin{tikzpicture}[scale=0.5, >=stealth]

        \def\xmin{-2}      
        \def\xmax{2}
        \def\ymin{-1.5}   
        \def\ymax{1.5}
        
        \def\scaleX{3}     
        \def\scaleY{3}

        \def\bcX{0.0}      
        \def\bcY{-0.3}     
        \def\rb{0.12}      
        
        \pgfmathsetmacro{\xL}{\xmin * \scaleX}
        \pgfmathsetmacro{\xR}{\xmax * \scaleX}
        \pgfmathsetmacro{\yB}{\ymin * \scaleY}
        \pgfmathsetmacro{\yT}{\ymax * \scaleY}
        
        \pgfmathsetmacro{\cX}{\bcX * \scaleX}
        \pgfmathsetmacro{\cY}{\bcY * \scaleY}
        \pgfmathsetmacro{\rB}{\rb * \scaleX} 
        \pgfmathsetmacro{\yFS}{0.0 * \scaleY} 

        \fill[blue!10] (\xL, \yB) rectangle (\xR, \yFS);
        
        \fill[gray!5] (\xL, \yFS) rectangle (\xR, \yT);
        
        \fill[gray!5, draw=red, thick] (\cX, \cY) circle (\rB);

        \draw[thick] (\xL, \yB) -- (\xL, \yT) -- (\xR, \yT) -- (\xR, \yB);
        
        \draw[ultra thick, line width=2pt] (\xL, \yB) -- (\xR, \yB);
        
        \draw[blue!60!black, ultra thick, dashed] (\xL, \yFS) -- (\xR, \yFS);

        \node[scale=0.9, font=\bfseries] at (-4, 2) {Air};
        \node[scale=0.9, font=\bfseries] at (-4, -2) {Water};
        
        \node[below, scale=0.8] at (\cX, \cY-\rB-0.2) {Bubble};

    \end{tikzpicture}
    \caption{Schematic of the initial configuration for the underwater explosion problem.}
     \label{fig:bubble_interface_initial}
 \end{figure}

\begin{figure}[H]
  \centering
  \begin{minipage}[b]{0.45\textwidth}
    \centering
    \includegraphics[width=\linewidth]{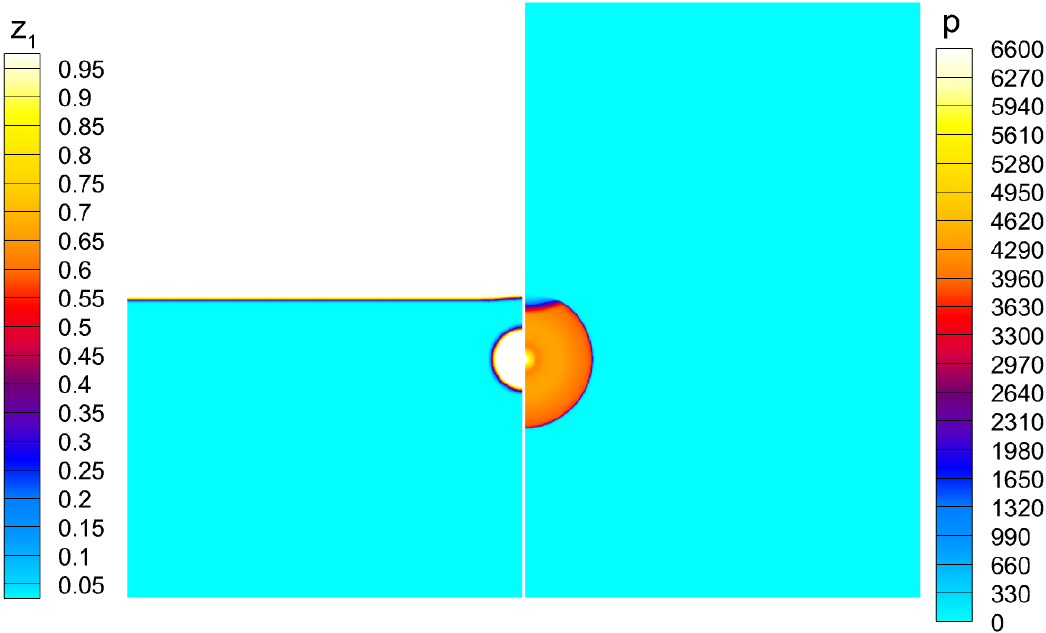}
    \caption*{$t=0.001$}
  \end{minipage}
  \hfill
  \begin{minipage}[b]{0.45\textwidth}
    \centering
    \includegraphics[width=\linewidth]{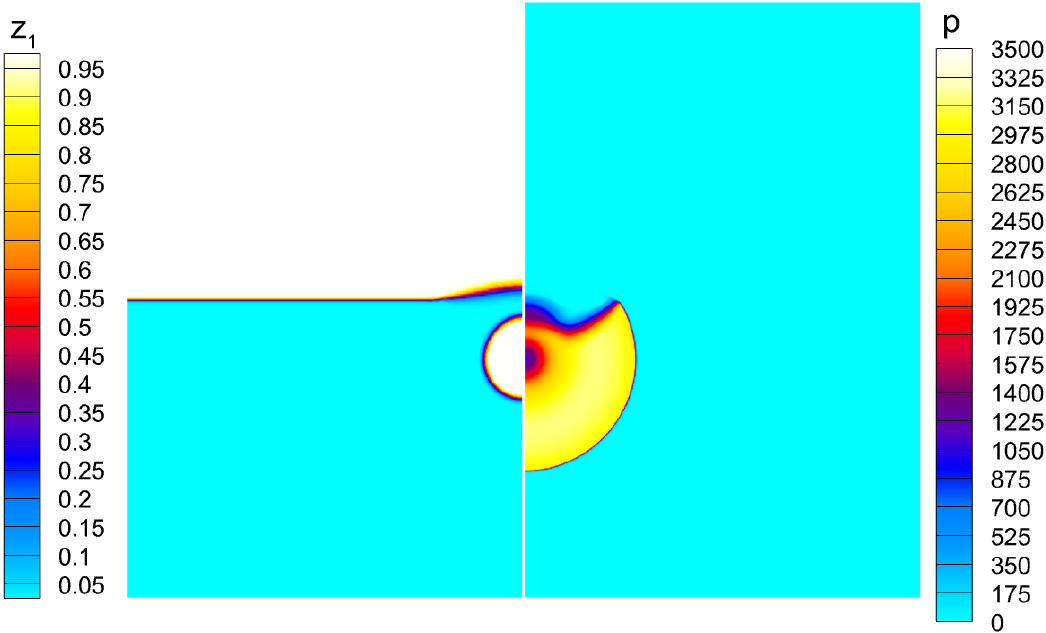}
    \caption*{$t=0.002$}
  \end{minipage}

  \par\medskip

  \begin{minipage}[b]{0.45\textwidth}
    \centering
    \includegraphics[width=\linewidth]{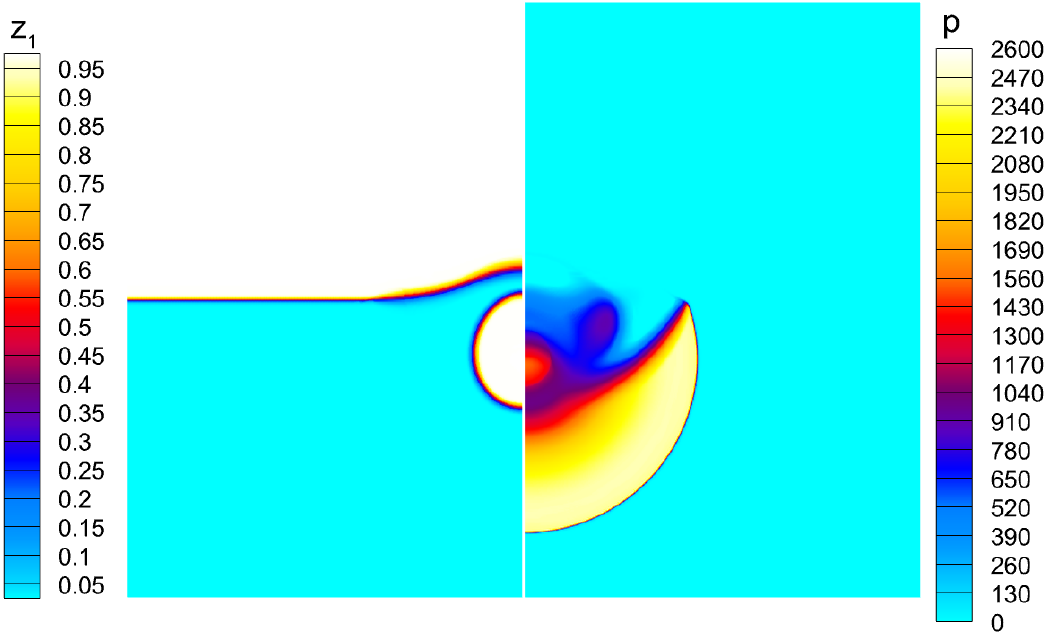}
    \caption*{$t=0.004$}
  \end{minipage}
  \hfill
  \begin{minipage}[b]{0.45\textwidth}
    \centering
    \includegraphics[width=\linewidth]{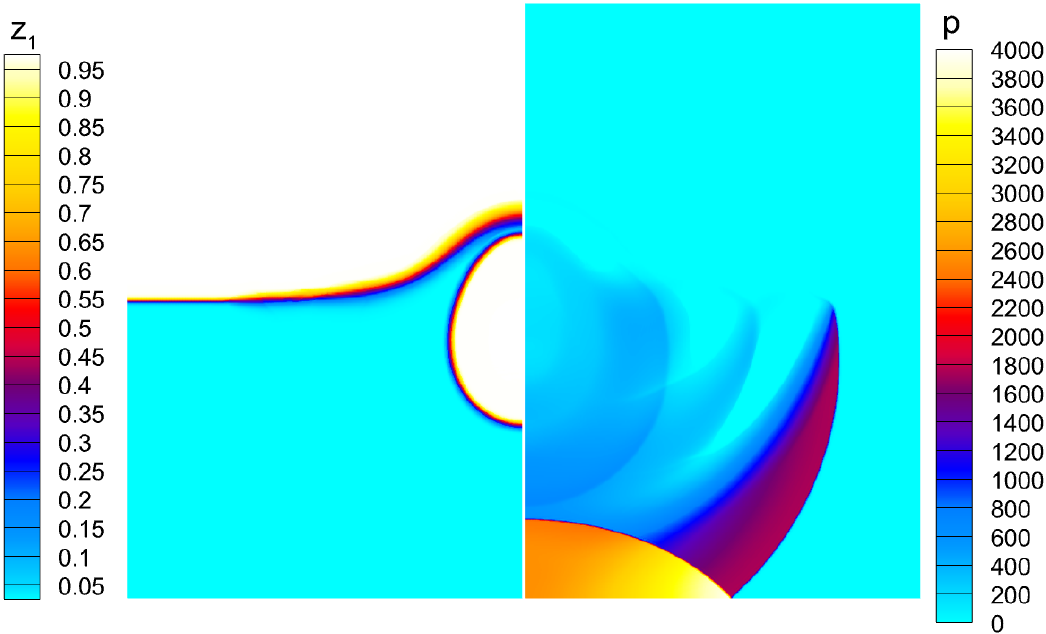}
    \caption*{$t=0.008$}
  \end{minipage}
  \caption{The pressure (right half of each figure) and the volume fraction (left half of each figure) are
shown at different times.}
  \label{fig:underwater_explosion}
\end{figure}

\subsection{Two-dimensional water shock-air bubble interaction problem}
This test case \cite{cheng2020quasi,xu2017explicit} considers the 
interaction between an air bubble and a planar shock 
wave propagating in water. The 
computational domain and the 
initial locations of the bubble and 
the incident shock wave are 
illustrated in Figure~\ref{fig:water_shock_bubble_initial}. 
The initial conditions are prescribed as follows
\begin{align*}
&(\rho_1, \rho_2, u, v, p, z_1) 
\\ &= 
\begin{cases} 
(0.0012, 1.0, 0, 0, 1.0, 1.0 - 10^{-6}), & \sqrt{(x-6)^2 + (y-6)^2} \le 3, \\
(0.0012, 1.325, -68.525, 0, 19153.0, 10^{-6}), & x \ge 11.4, \\
(0.0012, 1.0, 0, 0, 1.0, 10^{-6}), & \text{otherwise},
\end{cases}
\end{align*}
where the parameters are $\gamma_1 = 1.4$, $p_{w,1} = 0.0$ 
and $\gamma_2 = 4.4$, $p_{w,2} = 6000$. 
The computational domain is discretized using 
a uniform mesh consisting of $800 \times 800$ elements. 
Regarding the boundary conditions, solid wall conditions 
are enforced on the upper and lower boundaries. 
An outflow condition is applied at the left boundary, 
while the exact post-shock state is continuously prescribed at the right boundary. 
Figure~\ref{fig:water_shock_air_sequence} 
displays the computed density and volume 
fraction contours at $t = 0.015$, $0.02$, $0.025$, $0.03$, $0.035$, and $0.04$. 
Consistent with the numerical results reported in the 
literature \cite{chen2025generalized}, 
the major interfacial features and complex 
wave structures are accurately resolved 
by the BP-OEDG($P^2$) scheme without any spurious oscillations.
\begin{figure}[htbp]
    \centering
    \begin{tikzpicture}[scale=0.3, >=stealth]

        \def\side{12}      
        \def\bc{6}         
        \def\rb{3}         
        \def\xs{11.4}      

        \fill[blue!5] (0,0) rectangle (\side, \side);
        
        \fill[blue!20] (\xs, 0) rectangle (\side, \side);
        
        \fill[white, draw=black, thick] (\bc, \bc) circle (\rb);

        \draw[thick] (0,0) rectangle (\side, \side);
        
        \draw[red, ultra thick] (\xs, 0) -- (\xs, \side);
        
        \draw[->, ultra thick, red] (\xs+1.8, \bc) -- (\xs-0.9, \bc) 
            node[midway, right=15pt, red, scale=0.8] {Shock};

        \node[scale=0.8] at (\bc-1, \bc) {Air};
        \node[scale=0.8] at (\bc-1, \bc+4) {Water};        

        \draw[->] (0,-0.5) -- (\side+1, -0.5) node[right] {$x$};
        \node[below] at (0, -0.5) {$0$};
        \node[below] at (\bc, -0.5) {$6$};
        \node[below] at (\side, -0.5) {$12$};
        
        \draw[red, thick] (\xs, \side) -- (\xs, \side+0.2); 
        \node[above, red, scale=0.9, font=\bfseries] at (\xs, \side+0.2) {$11.4$};

        \draw[->] (-0.5, 0) -- (-0.5, \side+1) node[above] {$y$};
        \node[left] at (-0.5, 0) {$0$};
        \node[left] at (-0.5, \bc) {$6$};
        \node[left] at (-0.5, \side) {$12$};

        \draw[<->, dashed] (\bc, \bc) -- (\bc+\rb, \bc) node[midway, above, scale=0.7] {$R=3$};

    \end{tikzpicture}
    \caption{Initial configuration of shock-bubble interaction in water.}
    \label{fig:water_shock_bubble_initial}
\end{figure}

\begin{figure}[H]
  \centering
  \begin{minipage}[b]{0.31\textwidth}
    \centering
    \includegraphics[width=\linewidth]{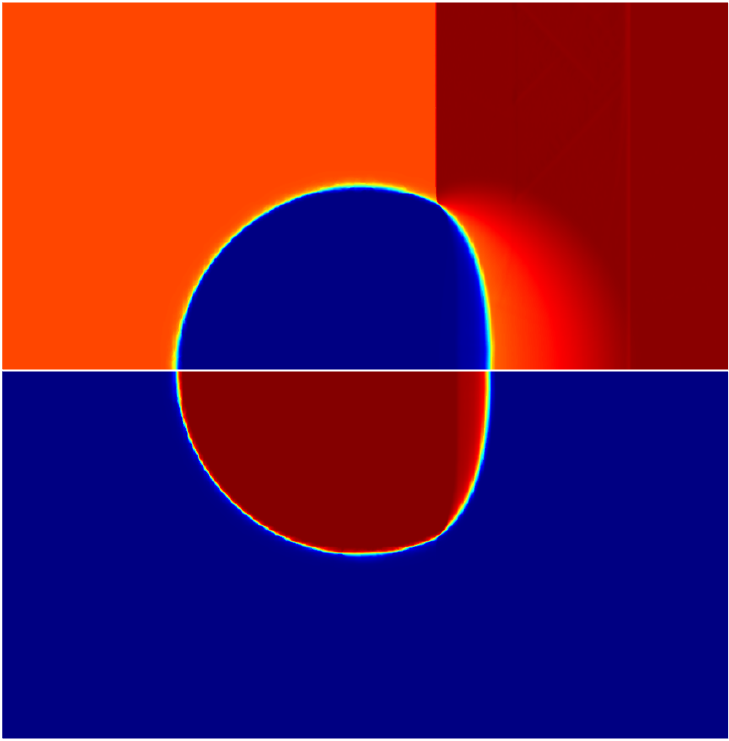}
    \caption*{$t=0.015$}
  \end{minipage}
  \hfill
  \begin{minipage}[b]{0.31\textwidth}
    \centering
    \includegraphics[width=\linewidth]{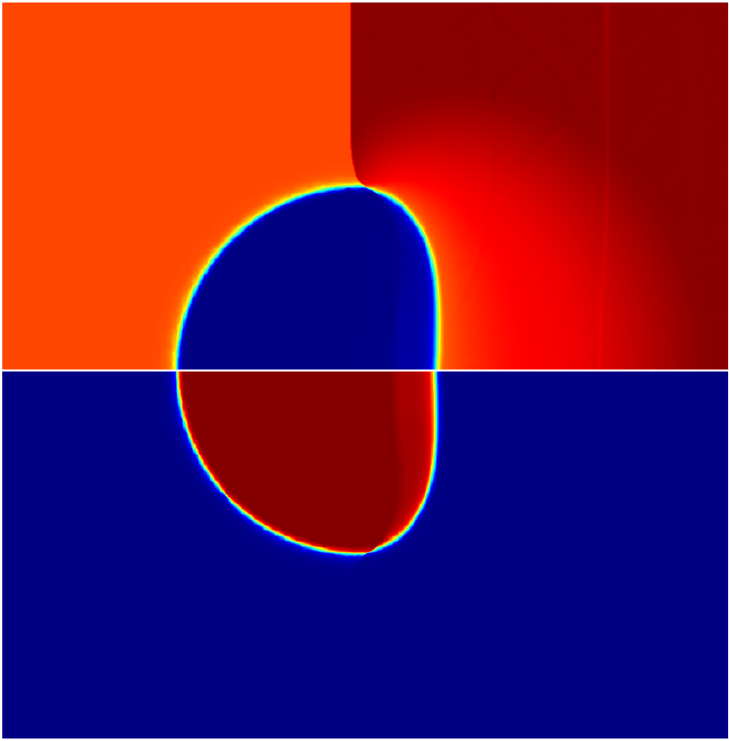}
    \caption*{$t=0.02$}
  \end{minipage}
  \hfill
  \begin{minipage}[b]{0.31\textwidth}
    \centering
    \includegraphics[width=\linewidth]{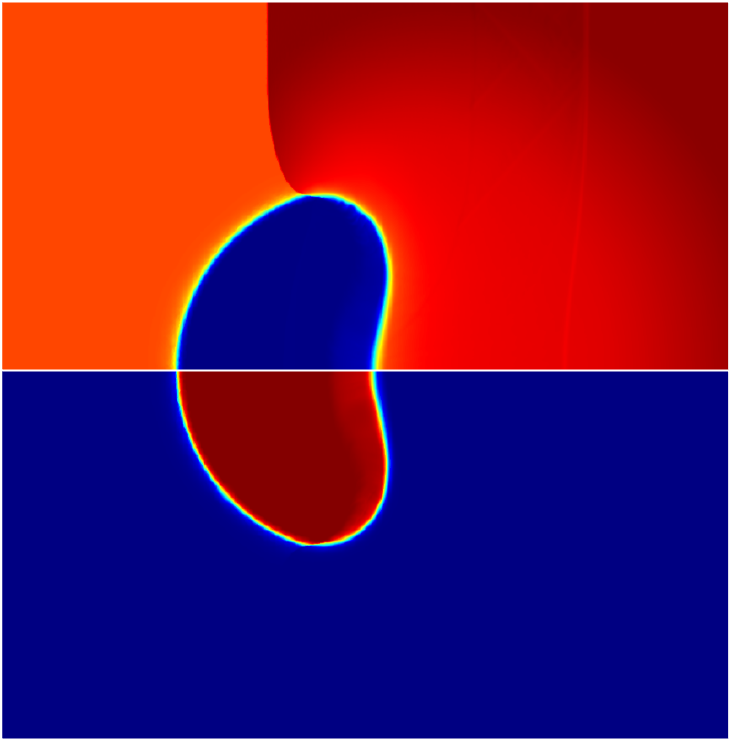}
    \caption*{$t=0.025$}
  \end{minipage}
  \par\medskip
  \begin{minipage}[b]{0.31\textwidth}
    \centering
    \includegraphics[width=\linewidth]{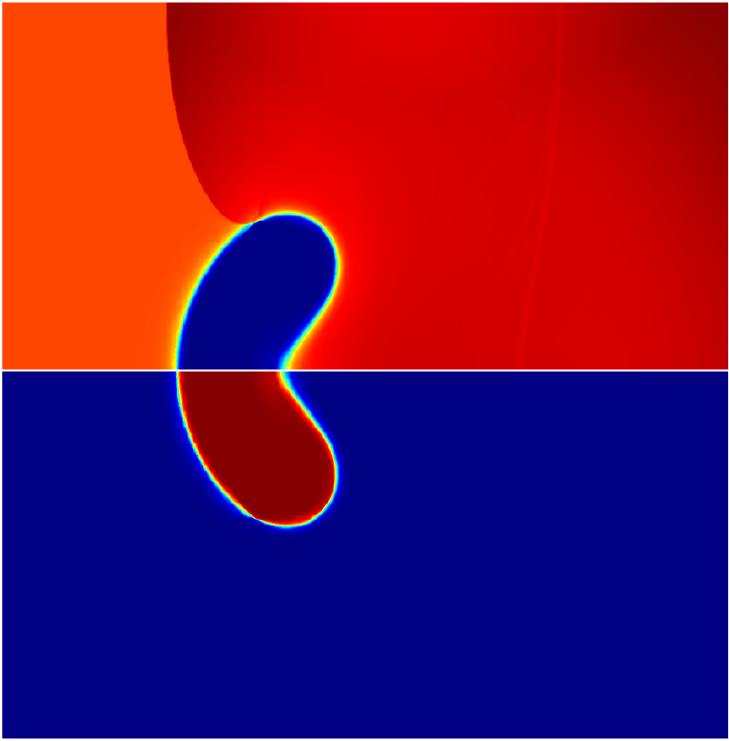}
    \caption*{$t=0.03$}
  \end{minipage}
  \hfill
  \begin{minipage}[b]{0.31\textwidth}
    \centering
    \includegraphics[width=\linewidth]{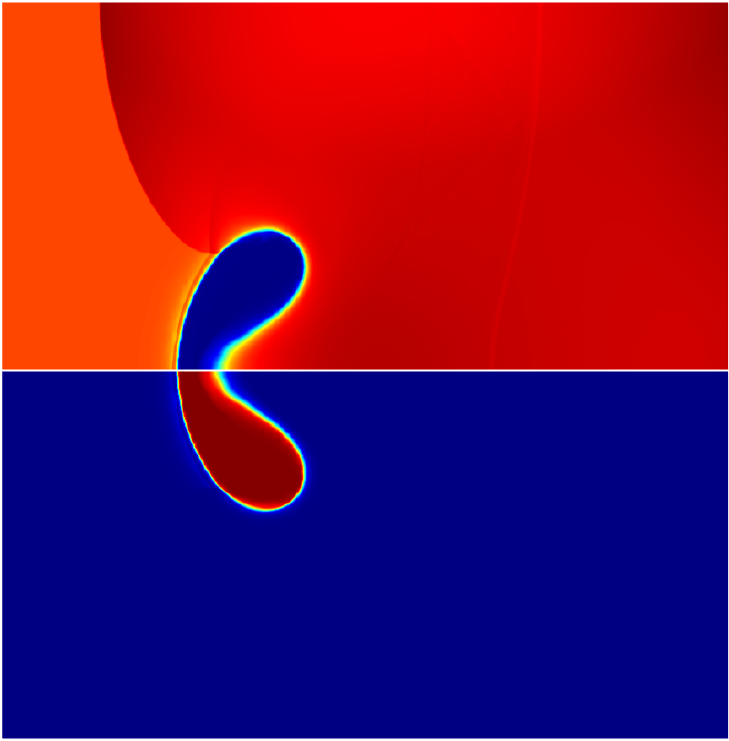}
    \caption*{$t=0.035$}
  \end{minipage}
  \hfill
  \begin{minipage}[b]{0.31\textwidth}
    \centering
    \includegraphics[width=\linewidth]{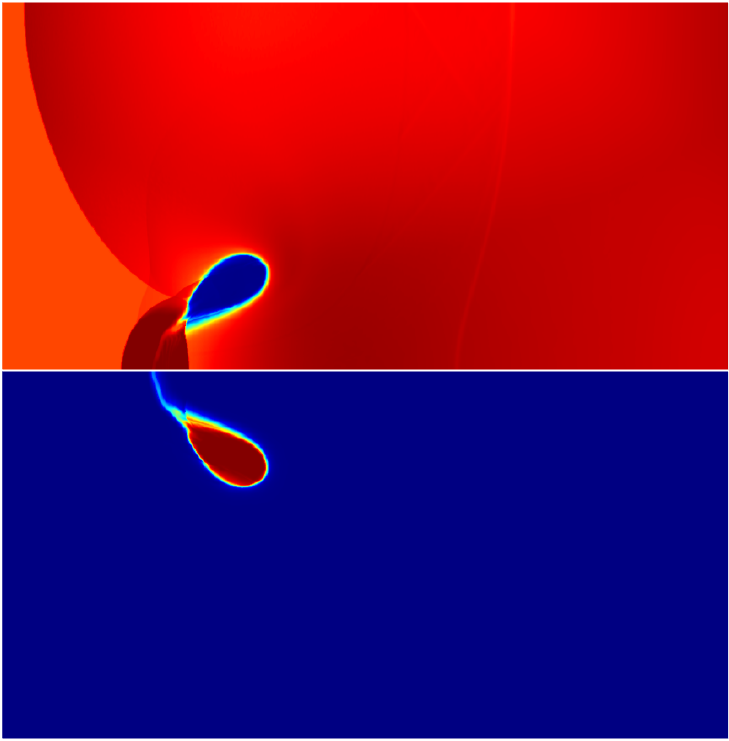}
    \caption*{$t=0.04$}
  \end{minipage}
  \caption{The mixture density (upper half of each figure) and the volume fraction (lower half of each figure) are
shown at different times.}
  \label{fig:water_shock_air_sequence}
\end{figure}

\section{Conclusion}\label{sec-conclusion}

In this paper, we propose a robust operator-splitting 
discontinuous Galerkin framework 
combined with an OEDG procedure and a BP limiter 
to solve Kapila's five-equation model for 
compressible two-phase flows. 
Crucially, by decoupling the stiff $\kappa$-source term, 
we introduce a novel adaptive implicit 
strategy hybridizing the backward Euler and 
second-order singly diagonally implicit Runge-Kutta schemes to 
resolve the stiffness-induced instabilities without severe time-step penalties. 
The proposed method not 
only mitigates the stiffness-induced instability 
associated with the $\kappa$ source term but also 
guarantees oscillation-free behavior, 
bound preservation, and strict satisfaction of the Abgrall condition. 
Numerical results validate its effectiveness in simulating gas-gas 
and gas-water two-phase flows under severe conditions. 
Future work will focus on overcoming the local order 
degradation issue \cite{sportisse2000analysis,descombes2004operator} 
induced by the operator splitting approach. 

\section{Acknowledgments}
This work was supported by the National Natural 
Science Foundation of China (Grant No. 52525102). 

The authors would like to express 
their gratitude to 
Prof. Ruifang Yan (Huazhong University of Science and Technology) and 
Dr. Xinyu Li (Xi'an Jiaotong University) for their valuable discussions.

\bibliographystyle{elsarticle-num} 
\bibliography{tex}
\end{document}